\documentclass[a4paper]{amsart} 

\usepackage{graphicx} 
\usepackage{mathrsfs} 
\usepackage{amsfonts} 
\usepackage[active]{srcltx} 
\usepackage{indentfirst,latexsym,bm,amsmath,pstricks,amssymb,amsthm,amscd} 
\usepackage[all,cmtip]{xy} 
\usepackage{makecell} 
\usepackage{appendix} 
\usepackage{svg} 
\usepackage{multirow} 
\usepackage[colorlinks=true,linkcolor=blue,anchorcolor=blue,citecolor=blue]{hyperref} 

\numberwithin{equation}{section} 

\usepackage{fancyhdr}
\allowdisplaybreaks 
\def\cE{\mathcal{E}}
\def\cF{\mathcal{F}}
\def\cG{\mathcal{G}}
\def\cI{\mathcal{I}}
\def\cO{\mathcal{O}}
\def\cZ{\mathcal{Z}}

\def\bC{\mathbb{C}}

\def\bP{\mathbb{P}}
\def\bQ{\mathbb{Q}}
\def\bR{\mathbb{R}}

\newcommand{\DF}[2]{{\displaystyle\frac{#1}{#2}}} 

\theoremstyle{plain}
\newtheorem{theorem}{Theorem}[section]
\newtheorem{lemma}[theorem]{Lemma}
\newtheorem{proposition}[theorem]{Proposition}
\newtheorem{corollary}[theorem]{Corollary}

\theoremstyle{definition}
\newtheorem{definition}[theorem]{Definition}
\newtheorem{example}[theorem]{Example}

\theoremstyle{remark}
\newtheorem{remark}[theorem]{Remark}

\newcommand{\bbm}{\begin{bmatrix}} 
\newcommand{\ebm}{\end{bmatrix}} 

\title[Canonical Models of Adjoint Foliated Structures on Surfaces]{Canonical Models of Adjoint Foliated Structures on Surfaces} 

\author{Jun Lu}
\author{Xiaohang Wu}
\author{Shi Xu$^{*}$}

\address[Jun Lu]{School of Mathematical Sciences, Key Laboratory of MEA(Ministry of Education) \& Shanghai Key Laboratory of PMMP, East China Normal University, Shanghai 200241, China}
\email{jlu@math.ecnu.edu.cn}

\address[Xiaohang Wu]{
School of Mathematics and Systems Science,  and Center for Mathematical Sciences,
Wuhan University of Science and Technology, Wuhan 430065, China  
}
\email{wuxiaoh001@foxmail.com}

\address[Shi Xu]{Yau Mathematical Sciences Center, Tsinghua University, Beijing, 100084, China}
\email{572463515@qq.com, shixumath@163.com}

\date{} 

\begin{document}

\footnotetext[1]{The corresponding author is Shi Xu.} 
\footnotetext[2]{
This work is supported by NSFC. The first author is also supported by the School of Mathematical Sciences, the Key Laboratory of MEA (Ministry of Education) $\&$ the Shanghai Key Laboratory of PMMP (No.~22DZ2229014), East China Normal University, Shanghai 200241, China.
}

\begin{abstract}
In this paper, we study adjoint foliated structures of the form $K_{\cF}+D$ on algebraic surfaces and their minimal and canonical models. 
We investigate the effective behavior of the linear systems $|m(K_{\cF}+D)|$ for sufficiently divisible $m>0$. 
As an application, we obtain an effective answer to a boundedness problem for foliated surfaces of general type posed by Hacon and Langer.
\end{abstract}

\maketitle 

\section{Introduction}
Let $(X,\cF)$ be a foliated surface over $\bC$, and let $K_{\cF}$ denote the canonical divisor of $\cF$.
If $K_{\cF}$ is not pseudo-effective, then $(X,\cF)$ is birationally equivalent to a rational fibration by a result of Miyaoka (cf.~\cite{Miya87,Bru15}).
If $K_{\cF}$ is pseudo-effective, McQuillan (cf.~\cite{MMcQ00,MMcQ08}) showed that $(X,\cF)$ admits a \emph{canonical model} $(X_{\rm c},\cF_{\rm c})$, constructed as follows:
\[
\xymatrix{
&(X',\cF')\ar^-{\sigma}[dr]\ar_-{\pi}[ld]&\\
(X,\cF)\ar@{-->}[rr]&&(X_{\rm c},\cF_{\rm c})
}
\]
where  
\begin{itemize}
\item $\pi$ is a resolution of $(X,\cF)$ such that $X'$ is a smooth projective surface and the induced foliation $\cF'$ has at most reduced (hence canonical) singularities;
\item $\sigma$ is a $K_{\cF'}$-non-positive birational contraction such that $(X_{\rm c},\cF_{\rm c})$ is a normal foliated surface with at most canonical singularities of $\cF_{\rm c}$ and $K_{\cF_{\rm c}}$ is nef.
Moreover, any irreducible curve $C\subset X_{\rm c}$ with $K_{\cF_{\rm c}}\cdot C=0$ satisfies $C^2\geq0$.
\end{itemize}

In particular, if $K_{\cF_{\rm c}}$ is big, then it is  \emph{numerically ample}, i.e., $K_{\cF_{\rm c}}^2>0$ and $K_{\cF_{\rm c}}\cdot C>0$ for any irreducible curve $C\subset X_{\rm c}$. 
However, $K_{\cF_{\rm c}}$ is not necessarily ample.
It is known that $K_{\cF_{\rm c}}$ fails to be $\bQ$-Cartier at \emph{cusp singularities}, whose resolutions are \emph{elliptic Gorenstein leaves} (egl's) (cf.~\cite[Theorem~IV.2.2]{MMcQ08}).
Moreover, $K_{\cF_{\rm c}}$ is ample if and only if $X_{\rm c}$ contains no cusp singularities (cf.~\cite[Corollary~IV.2.3]{MMcQ08}).

To address this issue, we study adjoint divisors of the form $K_{\cF}+D$, where $D$ is one of the following:
\begin{itemize}
\item[\rm (i)] $D = \epsilon K_X$ for some $\epsilon \in (0,1]$;
\item[\rm (ii)] $D = \Delta$, where $\Delta = \sum_{i=1}^l a_i C_i$ with $a_i \in [0,1]$ and each $C_i$ is not $\cF$-invariant.
\end{itemize}

Assume that $K_{\cF}+D$ is pseudo-effective, and write its Zariski decomposition as $K_{\cF}+D=P(D)+N(D)$. 
Following \cite[p.~886]{Sak84}, we have birational morphisms
\begin{equation}\label{seq:canonicalmodel}
\xymatrix{
(X,\cF,D)\ar^-{f}[r]&(X',\cF',D_{X'})\ar^-{h}[r]&(Y,\cG,D_Y),
}
\end{equation} 
where $f$ is a $(K_{\cF}+D)$-negative contraction and $g:=h\circ f$ is a $(K_{\cF}+D)$-non-positive contraction. 
We set $D_{X'}:=f_*D$ and $D_Y:=g_*D$, and call $(X',\cF',D_{X'})$ (resp.~$(Y,\cG,D_Y)$) the \emph{minimal model} (resp.~\emph{canonical model}) of $(X,\cF,D)$.
\medskip

We now state our main results.

\begin{theorem}[= Theorem~\ref{mainthm:<1/4} + Theorem~\ref{thm:Null} + Corollary~\ref{coro:null}(2)]\label{Thm:main1eKX}
Let $(X,\cF)$ be a log minimal foliated surface (cf.~Definition~\ref{def:logminfoliation}).
Assume $\epsilon\in(0,\frac{1}{4})\cap\bQ$ and that $K_{\cF}+\epsilon K_X$ is pseudo-effective. 
Let $K_{\cF}+\epsilon K_X = P(\epsilon)+N(\epsilon)$ be its Zariski decomposition. 
Then the following hold:
\begin{itemize}
\item[(1)] The negative part $N(\epsilon)$ is a disjoint union of all maximal $(\epsilon K_X,\cF)$-chains on $X$. 
In particular, $\lfloor N(\epsilon)\rfloor=0$. 

\item[(2)] Let $(Y,\cG,D_Y)$ be the canonical model of $(X,\cF,\epsilon K_X)$. 
Then  $\cG$ has at most log canonical singularities and 
 $Y$ has at most rational quotient singularities.
 In particular, $D_Y = \epsilon K_Y$, and the singularities of $Y$ are of the following types:
 \begin{itemize}
    \item[(a)] \emph{Cyclic quotient singularities}, whose minimal resolution is a chain of smooth rational $\cF$-invariant curves with self-intersection at most $-2$:
    \begin{equation}
    \xymatrix@=0.4cm{\mathop{\bigcirc}\limits_{-e_1} \ar@{-}[r]&\mathop{\bigcirc}\limits_{-e_2}\ar@{-}[r]&\cdots\cdots\ar@{-}[r]&\mathop{\bigcirc}\limits_{-e_r}}\notag
    \end{equation}
    Such chains are either $\cF$-chains (cf.~Theorem~\ref{thm:Null}(1)), or chains of $\cF$-invariant $(-2)$-curves $\Gamma$ with ${\rm Z}(\cF,\Gamma) = 2$, or chains of three $\cF$-invariant $(-2)$-curves $\Gamma_1 + \Gamma_2 + \Gamma_3$ where ${\rm Z}(\cF,\Gamma_1) = {\rm Z}(\cF,\Gamma_3) = 1$ and ${\rm Z}(\cF,\Gamma_2) = 3$.

\item[(b)] \emph{Dihedral quotient singularities}, whose minimal resolution has the following dual graph: 
  \begin{equation}
  \xymatrix@=0.4cm{
  \mathop{\bigcirc}\limits_{-2}\ar@{-}[rd]&&&\\
  &\mathop{\bigcirc}\limits_{-2}\ar@{-}[r]&\mathop{\bigcirc}\limits_{-2}\ar@{-}[r]&\cdots\cdots\ar@{-}[r]&\mathop{\bigcirc}\limits_{-2}\\
  \mathop{\bigcirc}\limits_{-2}\ar@{-}[ru]&&&
  }\notag
  \end{equation}
  Here, each component is an $\cF$-invariant $(-2)$-curve. 
  Among them, the two leftmost curves $\Gamma$ satisfy ${\rm Z}(\cF,\Gamma) = 1$, the central curve $\Gamma'$ satisfies ${\rm Z}(\cF,\Gamma') = 3$, and each of the remaining curves $\Gamma''$ satisfies ${\rm Z}(\cF,\Gamma'') = 2$.
  \end{itemize}
\end{itemize}
\end{theorem}

In Theorem~\ref{Thm:main1eKX}, the morphism $X\to Y$ is the minimal resolution of $Y$. 
A $(D,\cF)$-chain means an $\cF$-chain satisfying additional conditions depending on $D$ (see Definition~\ref{def:DFchain}).
\medskip

Define
\begin{equation}
  \alpha(D,\cF):=\left\lfloor\frac{\left(K_X\cdot A+2\right)^2}{4A^2}+\frac{9A^2+6K_X\cdot A}{4}\right\rfloor
  \end{equation}
  where $A:=i(D,\cF)\, P(D)$ and $i(D,\cF)$ is the smallest positive integer $m$ such that $m P(D)$ is integral.
  For convenience, when $D=0$ we write $i(\cF):=i(0,\cF)$ and $\alpha(\cF):=\alpha(0,\cF)$.

\begin{corollary}[= Theorem~\ref{them:boundednessofD} + Corollary~\ref{coro:boundednessofeKX}]\label{coro:main2eKX}
Under the assumption and notation in Theorem~\ref{Thm:main1eKX}, if $K_{\cF}+\epsilon K_X$ is big, then $K_{\cG}+\epsilon K_Y$ is ample.
More precisely, the divisor
$$\big(i(\epsilon K_X,\cF)\left(\alpha(\epsilon K_X,\cF)+3\right)\big)\, (K_{\cG}+\epsilon K_Y)$$
is very ample.
Furthermore, if $K_{\cF}$ is pseudo-effective, then there exists a positive integer $\mathfrak{n}$, depending only on $\epsilon$, $i(\cF)$ and ${\rm Vol}(K_{\cF}+\epsilon K_X)$, such that
the divisor $\mathfrak{n} (K_{\cG}+\epsilon K_Y)$ is very ample.
\end{corollary}

\begin{remark}
\begin{itemize}
\item[(1)] The bound $0<\epsilon<\tfrac{1}{4}$ is sharp. 
Indeed, when $\epsilon=\tfrac{1}{4}$, the foliation $\cG$ may fail to be log canonical (see Example~\ref{ex:e=1/4}).
\item[(2)] Theorems~\ref{Thm:main1eKX} and Corollary~\ref{coro:main2eKX} can be compared with \cite[Theorems~1.1(2) and~1.2]{SS23}, where the $(K_{\cF}+\epsilon K_X)$-MMP is studied for $\epsilon \in (0,\tfrac{1}{5})$. 
Our results differ slightly. 
A natural extension is to the case where $(X,\cF)$ is a minimal foliated surface (cf.~Definition~\ref{def:minfoliation}); however, in this setting the negative part $N(\epsilon)$ becomes more involved and will be treated elsewhere.
\end{itemize}
\end{remark}

\begin{theorem}[= Theorem~\ref{mainthm:<1/4} + Theorem~\ref{thm:Null} + Corollary~\ref{coro:null}(1)]\label{Thm:main3Delta}
Let $(X,\cF,\Delta)$ be a foliated triple, where 
\begin{itemize}
\item[(i)] $\cF$ is a foliation on $X$ with at most canonical singularities;
\item[(ii)] $\Delta=\sum_{i=1}^la_iC_i$, where $a_i\in[0,1)$ and each $C_i$ is not $\cF$-invariant;  
\item[(iii)] $X$ is a smooth projective surface containing no $\cF$-invariant $(-1)$-curves $E$ satisfying
$K_{\cF}\cdot E\leq0$ and $K_{\cF}\cdot E+2\Delta\cdot E\leq1.$
\end{itemize}
Assume that $K_{\cF}+\Delta$ is pseudo-effective, with the Zariski decomposition $K_{\cF}+\Delta=P(\Delta)+N(\Delta)$.
Then the following hold:
\begin{itemize}
\item[(1)] The negative part $N(\Delta)$ is a disjoint union of all maximal $(\Delta,\cF)$-chains on $X$. 
In particular, $\lfloor N(\Delta)\rfloor=0$.

\item[(2)] Let $(Y,\cG,\Delta_Y)$ be the canonical model of $(X,\cF,\Delta)$. 
Then  $\cG$ has at most canonical singularities and the singularities of $Y$ are of the following types:
\begin{itemize}
  \item[(a)] \emph{Cyclic quotient singularities}, whose minimal resolution is a chain of smooth rational $\cF$-invariant curves with self-intersection at most $-2$:
  \begin{equation}
    \xymatrix@=0.4cm{\mathop{\bigcirc}\limits_{-e_1} \ar@{-}[r]&\mathop{\bigcirc}\limits_{-e_2}\ar@{-}[r]&\cdots\cdots\ar@{-}[r]&\mathop{\bigcirc}\limits_{-e_r}}\notag
    \end{equation}
    Such chains are either $\cF$-chains (cf.~Theorem~\ref{thm:Null}(1)),
chains of smooth rational $\cF$-invariant curves with self-intersection at most $-2$ and disjoint from $\Delta$,
or chains of three smooth rational $\cF$-invariant curves $\Gamma_1 + \Gamma_2 + \Gamma_3$, also disjoint from $\Delta$,
where $\Gamma_1^2 = \Gamma_3^2 = -2$,
${\rm Z}(\cF, \Gamma_1) = {\rm Z}(\cF, \Gamma_3) = 1$,
$\Gamma_2^2 \leq -2$,
and ${\rm Z}(\cF, \Gamma_2) = 3$.

  \item[(b)] \emph{Dihedral quotient singularities}, whose resolution has the following dual graph:
  \begin{equation}
    \xymatrix@=0.4cm{
    \mathop{\bigcirc}\limits_{-2}\ar@{-}[rd]&&&\\
    &\mathop{\bigcirc}\limits_{-e_3}\ar@{-}[r]&\mathop{\bigcirc}\limits_{-e_4}\ar@{-}[r]&\cdots\cdots\ar@{-}[r]&\mathop{\bigcirc}\limits_{-e_r}\\
    \mathop{\bigcirc}\limits_{-2}\ar@{-}[ru]&&&
    }\notag
    \end{equation}
    Here, each component is a smooth rational $\cF$-invariant curve of self-intersection at most $-2$ and disjoint from $\Delta$. 
    Among them, the two leftmost curves $\Gamma$ satisfy ${\rm Z}(\cF,\Gamma) = 1$ and $\Gamma^2=-2$, the central curve $\Gamma'$ satisfies ${\rm Z}(\cF,\Gamma') = 3$, and each of the remaining curves $\Gamma''$ satisfies ${\rm Z}(\cF,\Gamma'') = 2$.
    
  \item[(c)] \emph{Cusp singularities disjoint from $\Delta$}, whose resolution is a cycle of smooth rational $\cF$-invariant curves with self-intersection at most $-2$, all disjoint from $\Delta$
\begin{equation}
\xymatrix@=0.4cm{
&\mathop{\bigcirc}\ar@{-}[r]&\mathop{\bigcirc}\ar@{-}[rd]&\\
\mathop{\bigcirc}\ar@{-}[ur]\ar@{-}[dr]&&&\mathop{\bigcirc}\ar@{--}[dl]\\
&\mathop{\bigcirc}\ar@{-}[r]&\mathop{\bigcirc}&
}\notag
\end{equation}
or a nodal rational $\cF$-invariant curve of negative self-intersection, disjoint from $\Delta$.
\end{itemize}
\end{itemize}
\end{theorem}

\begin{remark}
\begin{itemize}
\item[(1)] When $\Delta=0$, Theorem~\ref{Thm:main3Delta} recovers \cite[Proposition~III.2.1 and Theorem~1.III.3.2]{MMcQ08}.

\item[(2)] One can choose a divisor $\Delta$ as in Theorem~\ref{Thm:main3Delta} such that $\Delta$ intersects every elliptic Gorenstein leaf.
In this case, the divisor $K_{\cG} + \Delta_Y$ is ample.

\item[(3)] The condition $0 \le \mathrm{Coeff}(\Delta) < 1$ is sharp. 
Indeed, let $\Delta=C$ be an irreducible non-$\cF$-invariant curve with 
$C^2<0$ and $\mathrm{tang}(\cF,C)=0$. 
Then the contraction produces a point $p \in Y$ which is not canonical (but log canonical) for $\cG$. 
Moreover, the minimal resolution of $p$ corresponds to case~(5) in Theorem~\ref{thm:NullCtang=0}; 
in particular, it also appears as case~6 or~7 in \cite[Theorem~1.1]{YAChen23}.

\end{itemize}
\end{remark}

\begin{proposition}[= Theorem \ref{them:boundednessofD} + Proposition \ref{prop:P(D)H1}]
  Under the assumptions and notation in Theorem~\ref{Thm:main3Delta}, assume further that $K_{\cF}+\Delta$ is big.
  Suppose $m$ is divisible by $i(\Delta,\cF)$. Then we have the following results.
\begin{itemize}
\item[(1)]If $m\geq i(\Delta,\cF)\left(\alpha(\Delta,\cF)+1\right)$, then
$H^i(X,mP(\Delta))=0$ for $i=1,2$, and 
\begin{equation}
\dim H^0(m(K_{\cF}+\Delta))=\frac{m^2}{2}{\rm Vol}(K_\cF+\Delta)-\frac{m}{2}K_X\cdot P(\Delta)+\chi(\cO_X).\notag
\end{equation} 
\item[(2)] If $m\geq i(\Delta,\cF)\left(\alpha(\Delta,\cF)+3\right)$, 
then $|m(K_{\cG}+\Delta_Y)|$ induces a birational map which is an isomorphism on the complement of the cusp singularities.
\end{itemize}
\end{proposition}

\begin{corollary}\label{coro:HLconjecture}
  Let $(X,\cF)$ be a minimal foliated surface (cf.~Definition \ref{def:minfoliation}).  
  Assume $K_{\cF}$ is big with the Zariski decomposition $K_{\cF}=P+N$.
    Let  $(X_{\rm c},\cF_{\rm c})$ be the canonical model of $(X,\cF)$.
    Suppose $m$ is divisible by $i(\cF)$.
    Then we have:
    \begin{itemize}
    \item[\rm(1)] 
      If $m\geq i(\cF)\left(\alpha(\cF)+1\right)$, then 
      $H^1(mP)=H^2(mP)=0$ and 
      \begin{equation}
      \dim H^0(mK_{\cF})=\frac{m^2}{2}{\rm Vol}(K_\cF)-\frac{m}{2}K_X\cdot P+\chi(\cO_X).\notag
      \end{equation}
    \item[\rm(2)] 
    If $m\geq i(\cF) \left(\alpha(\cF)+3\right)$, then $|mK_{\cF_{\rm c}}|$ induces a birational map which is an isomorphism on the complement of the cusp singularities.
    \end{itemize}
    \end{corollary}
\begin{remark}
  Serrano  proved that a fibred surface (i.e., algebraically integrable foliation) has no egl's (cf.~\cite{Ser92,Ser93}). 
  In this case, the divisor $mK_{\cF_{\rm c}}$ is very ample for $m\geq i(\cF)\left(\alpha(\cF)+3\right)$ divisible by $i(\cF)$.
  Moreover, when the foliation has no cusp singularities (or egl's), Corollary \ref{coro:HLconjecture} essentially coincides with \cite[Theorem 7.2]{Tan24}, up to minor modifications.
\end{remark}

Let $i_{\bQ}(\cF_{\rm c})$ be the smallest positive integer $m$ such that $mK_{\cF_{\rm c}}$ is Cartier at all $\bQ$-Gorenstein points of $\cF_{\rm c}$. Then
\[
i(\cF)= i_{\bQ}(\cF_{\rm c}),\quad 
K_X\cdot P=K_{X_{\rm c}}\cdot K_{\cF_{\rm c}},\quad 
P^2=K_{\cF_{\rm c}}^2,
\]
where the first equality follows from \cite[Theorem~(4.2)]{Sak84} and \cite[Theorem~(1.7)]{Artin62}. 
Hence $\mathfrak{m}:=i(\cF)\big(\alpha(\cF)+3\big)$ depends only on the Hilbert function $P(m)=\chi(mK_{\cF_{\rm c}})$ (cf.~\cite[Proposition~4.1]{HA21}). 
Therefore, Corollary~\ref{coro:HLconjecture}(2) gives an effective answer to \cite[Conjecture~1]{HA21}, whose existence statement was previously established in \cite{YAChen21}.
Moreover, \cite{Pas24} shows that the set of Hilbert functions $P(m)=\chi(mK_{\cF_{\rm c}})$ of $2$-dimensional foliated canonical models with fixed 
$i_{\bQ}(\cF_{\rm c})$, ${\rm Vol}(\cF_{\rm c})$, and $K_{X_{\rm c}}\cdot K_{\cF_{\rm c}}$ is finite.
\medskip

Finally, we present the following variant of Theorem~\ref{Thm:main1eKX}.

\begin{proposition}[= Corollary~\ref{cor:eKXcanonical}]\label{prop:maineKXDuval}
Let $(X,\cF)$ be a minimal foliated surface (cf.~Definition~\ref{def:minfoliation}).
Assume that $K_{\cF}$ is big with the Zariski decomposition $K_{\cF}=P+N$, and 
$$\epsilon\in\left(0,\frac{1}{3\cdot i(\cF)}\right)\cap\bQ.$$
Then $K_{\cF}+\epsilon K_X$ is big with the Zariski decomposition $P(\epsilon)+N(\epsilon)$, where
\begin{itemize}
\item[(1)] ${\rm Supp}\,N(\epsilon)={\rm Supp}\,N$ is the disjoint union of all maximal $\cF$-chains on $X$.
\item[(2)] ${\rm Null}\,P(\epsilon)$ consists of the following curves:
\begin{itemize}
\item[(i)] components of $\cF$-chains;
\item[(ii)] $(-2)$-curves $C$ with $P\cdot C=0$ that are not contained in any $\cF$-chain.
\end{itemize}
\end{itemize}
\end{proposition}
  
Let $(Y,\cG,\epsilon K_Y)$ be the canonical model of $(X,\cF,\epsilon K_X)$. 
Then $(Y,\cG)$ coincides with the \emph{minimal partial du Val resolution} of $(X_{\rm c},\cF_{\rm c})$ introduced in \cite{YAChen21}.
By Corollary~\ref{cor:eKXcanonical}(3), ${\rm Vol}(K_{\cF}+\epsilon K_X)$ is bounded in terms of ${\rm Vol}(K_{\cF})$, $i(\cF)$, $\epsilon$, and $K_X\cdot P$.
Arguing as in Corollary~\ref{coro:main2eKX}, there exists $\mathfrak{N}>0$, depending only on $i(\cF)$, $\epsilon$, $K_X\cdot P$, and ${\rm Vol}(K_{\cF})$, such that $\mathfrak{N}(K_{\cG}+\epsilon K_Y)$ is very ample. 
Moreover, $\mathfrak{N}$ depends only on $\epsilon$ and the Hilbert function $P(m)=\chi(mK_{\cF_{\rm c}})$.      
  
\section{Preliminaries}
\subsection{Definition of foliations}
A \emph{foliation} $\cF$ on a normal surface $X$ is a rank $1$ saturated subsheaf
$T_{\cF}\subset T_X$ of the tangent sheaf.
Since $T_X\simeq{\rm Hom}_{\cO_X}(\Omega_X,\cO_X)$ is reflexive, $T_{\cF}$ is also reflexive.
We define the \emph{canonical divisor} $K_{\cF}$ of $\cF$ by
$\cO_X(-K_{\cF})\cong T_{\cF}$.

Let $f:Y\to X$ be a proper birational morphism and let $\cF$ be a foliation on $X$.
The \emph{pullback foliation} $f^*\cF$ is defined as the saturation of the kernel of the composition
\[
T_Y\to f^*T_X\to f^*(T_X/T_{\cF}).
\]
Conversely, if $\cG$ is a foliation on $Y$, we define the \emph{pushforward foliation} $f_*\cG$ 
as the saturation of the image of the composition
\[
f_*T_{\cG}\to f_*T_Y\to (f_*T_Y)^{**}=T_X.
\]
Note that $f^*f_*\cG=\cG$ and $f_*f^*\cF=\cF$ by \cite[Lemma~1.8]{HA21}.
\medskip

In this paper, we mainly work on a smooth surface $X$, unless stated otherwise.

\subsection{Singularities of foliations}
A point $p\in X$ is a \emph{singular point} of $\cF$ if either $p$ is a singular point of $X$ 
or the quotient sheaf $T_X/T_{\cF}$ is not locally free at $p$.
We denote by ${\rm Sing}(\cF)$ the set of singular points of $\cF$.

\subsubsection{Reduced singularities}
Let $p$ be a singular point of $\cF$ such that $X$ is smooth at $p$.
In local coordinates $(x,y)$ centered at $p$, the foliation is locally defined by
\begin{equation}
\nu
=
a(x,y)\frac{\partial}{\partial x}
+
b(x,y)\frac{\partial}{\partial y},
\quad
\text{equivalently}\quad
\omega
=
a(x,y)\,{\rm d}y
-
b(x,y)\,{\rm d}x .
\end{equation}

Let $\lambda_1,\lambda_2$ be the eigenvalues of the linear part $(D\nu)(p)$ of $\nu$ at $p$.
The singularity $p$ is called \emph{non-degenerate} 
if both eigenvalues $\lambda_1,\lambda_2$ are nonzero.
It is called  \emph{reduced} 
if one of the two eigenvalues, say $\lambda_2$, is nonzero 
and the quotient 
$\lambda=\lambda_1/\lambda_2$
is not a positive rational number. 
In particular, if $\lambda=0$, we call $p$ a \emph{saddle-node}.
A foliation $\cF$ is said to be \emph{reduced} 
if all its singular points are reduced.

\begin{remark}
If $p$ is a saddle-node, then after a suitable change of coordinates,
\[
\nu=
\left(x+axy^k+yF(x,y)\right)\tfrac{\partial}{\partial x}
+
y^{k+1}\tfrac{\partial}{\partial y},
\]
where $a\in\bC$, $k\in\mathbb{Z}^{>0}$, and 
$F$ is a holomorphic function which vanishes at $(0,0)$ up to order $k$.
The curve $(y=0)$ is the \emph{strong separatrix}.
If $F=0$, then $(x=0)$ is called the \emph{weak separatrix}. (See \cite[p.~3]{Bru15}.)
\end{remark}

Let $\sigma:X'\to X$ be the blow-up of $X$ at $p$ with exceptional divisor $E$, and let $\cF'=\sigma^*\cF$. 
Then
\begin{equation}\label{equ:KF-blow-up}
K_{\cF'} = \sigma^*K_{\cF} + (1-l(p))E,
\end{equation}
where $l(p)$ denotes the vanishing order of $\pi^*(\omega)$ along $E$. 
Let $a(p)$ denote the vanishing order of $\omega$ at $p$. 
Then $l(p)=a(p)$ if $E$ is $\cF$-invariant, and $l(p)=a(p)+1$ otherwise.
(See \cite[p.~16--17]{Bru15})

\begin{theorem}[Seidenberg]\label{thm:Seidenberg}
Given any foliated surface $(X, \cF)$, 
there exists a sequence of blow-ups
$\sigma:X'\to X$ 
such that the induced foliation $\cF'$ on $X'$ is reduced.
(See \cite{Sei68} or \cite[Theorem 1.1]{Bru15}.)
\end{theorem}

\subsection{Index theorems}
We recall several index theorems for foliations on surfaces following \cite{Bru15}.

Let $\cF$ be a foliation on a smooth surface $X$.
A curve $C\subseteq X$ is said to be \emph{$\cF$-invariant} if 
the inclusion $T_{\cF}|_C\to T_X|_C$ factors through $T_C$, where $T_C$ is the tangent bundle of $C$.

\subsubsection{Non-invariant curves}

For a non-$\cF$-invariant curve $C$ and $p\in C$, define
\[
{\rm tang}(\cF,C,p)
:=
\dim_{\bC}\frac{\cO_p}{\langle f,\nu(f)\rangle},
\]
where $\nu$ is the local generator of $\cF$ at $p$ and $f=0$ is the local equation of $C$.

\begin{lemma}
  We have
  \begin{equation}\label{tang}
  {\rm tang}(\cF,C):=\sum_{p\in C}{\rm tang}(\cF,C,p)=K_{\cF}\cdot C+C^2\, (\geq0).
  \end{equation} 
  (See \cite[Proposition 2.2]{Bru15}.)
\end{lemma}    

\subsubsection{Invariant curves}
Let $\omega$ be a local $1$-form defining $\cF$ in a neighborhood of $p\in C$,
where $C$ is an $\cF$-invariant curve locally defined by $f=0$.  
We can write
\[
g\,\omega=h\,{\rm d}f+f\,\eta,
\]
where $\eta$ is a holomorphic $1$-form, $g$ and $h$ are holomorphic functions, and $h$ and $f$ are relatively prime. 

We define 
\begin{align}
  {\rm Z}(\cF,C,p)&:=\text{vanishing order of $\left.\frac{h}{g}\right|_C$ at $p$}\\
  {\rm CS}(\cF,C,p)&:={\rm Res}_{p}\left\{\left.-\frac{\eta}{h}\right|_C\right\}.
\end{align}
The index ${\rm Z}$ is the Gomez-Mont--Seade--Verjovsky index
and ${\rm CS}$ is the residue-type index (cf.~\cite{Bru15,Bru97,CS82,GSV91}). 
Both indices vanish if $p$ is not a singularity of $\cF$.

\begin{lemma}\label{lem:csformula}
For an $\cF$-invariant curve $C$, we have
\begin{align}
{\rm Z}(\cF,C)&:= \sum_{p\in C} {\rm Z}(\cF,C,p)=K_{\cF}\cdot C+2-2p_a(C),\label{equ:Z}\\
{\rm CS}(\cF,C)& := \sum_{p\in C} {\rm CS}(\cF,C,p)=C^2,\label{equ:CS}
\end{align} 
where $p_a(C)$ denotes the arithmetic genus of $C$, and the second equality is the \emph{Camacho-Sad formula}.
 (See \cite[Proposition~2.3 and Theorem~3.2]{Bru15}.)
\end{lemma}

\begin{lemma}\label{lem:Z-CS-index-for-reduced}
Let $p$ be a reduced singularity of the foliation $\cF$.
\begin{itemize}
  \item[(1)] If $p$ is non-degenerate and $\cF$ locally around $p$ defined by $\omega=\lambda y(1+o(1)){\rm d}x-x(1+o(1)){\rm d}y$, then
  \begin{equation}
  \begin{cases}
  {\rm CS}(\cF,x=0,p)=\tfrac{1}{\lambda},\quad   {\rm CS}(\cF,y=0,p)=\lambda,\quad   {\rm CS}(\cF,xy=0,p)=\lambda+\tfrac{1}{\lambda}+2,\\
  {\rm Z}(\cF,x=0,p)={\rm Z}(\cF,y=0,p)=1,\quad   {\rm Z}(\cF,xy=0,p)=0.
  \end{cases}\notag
  \end{equation}
  \item[(2)] If $p$ is a saddle-node and $\cF$ locally around $p$ defined by $\omega=y^{k+1}{\rm d}x-(x(1+ay^k)+y\,o(k)){\rm d}y$, where $k\in\mathbb{Z}^{>0}$ and $a\in\mathbb{C}$,
   then 
  \[
  {\rm CS}(\cF,y=0,p)=0,\quad\text{and}\quad {\rm Z}(\cF,y=0,p)=1.
  \]
  If furthermore there exists a weak separatrix $(x=0)$, then 
    \[
  {\rm CS}(\cF,x=0,p)=a,\quad\text{and}\quad {\rm Z}(\cF,x=0,p)=k+1\,(\geq2).
  \]
\end{itemize}
(See \cite[p.~30--31]{Bru15}.)
\end{lemma}

\begin{theorem}[Separatrix theorem]\label{thm:separatrix}
Let $\cF$ be a foliation on a smooth projective surface $X$ and let $C\subset X$ be a connected compact $\cF$-invariant curve such that:
\begin{itemize}
\item[(i)] All the singularities of $\cF$ on $C$ are reduced (in particular, $C$ has only normal crossing singularities);
\item[(ii)] If $C_1,\cdots,C_n$ are the irreducible components of $C$, then the intersection matrix $(C_i\cdot C_j)_{1\leq i,j\leq n}$ is negative definite and the dual graph $\Gamma$ is a tree.
\end{itemize}
Then there exists at least one point $p\in C\cap {\rm Sing}(\cF)$ and a separatrix through $p$ not contained in $C$.
(See \cite[Theorem 3.4]{Bru15}.)  
\end{theorem}

\subsection{(Log) canonical singularities of foliations}
Let $f:Y \to X$ be a proper birational morphism between normal surfaces, 
and let $\cF$ be a foliation on $X$ with pullback foliation $f^*\cF$ on $Y$.  
We define the \emph{discrepancy} of $\cF$ along a prime divisor $E$ on $Y$ 
as in \cite[Definition~I.1.2, Definition~I.1.5, and Fact~I.2.15]{MMcQ08} by
\[
a(E,\cF) := {\rm ord}_E\bigl(K_{f^*\cF}-f^*K_{\cF}\bigr).
\]

We say that $(X,\cF)$ is \emph{terminal} (resp. \emph{canonical}) 
if $a(E,\cF) > 0$ (resp. $\geq 0$) for every exceptional divisor $E$ over $X$, 
and $(X,\cF)$ is \emph{log terminal} (resp. \emph{log canonical}) 
if $a(E,\cF) > -\iota(E)$ (resp. $\geq -\iota(E)$) for every exceptional divisor $E$ over $X$, 
where
\begin{equation}
\iota(E):=
\begin{cases}
0,& \text{if $E$ is $f^*\cF$-invariant},\\[1mm]
1,& \text{otherwise}.
\end{cases}
\end{equation}

\begin{definition}\label{def:minfoliation}
A foliated surface $(X,\cF)$ is called \emph{minimal} if
\begin{itemize}
\item[(i)] $\cF$ has at most canonical singularities, and
\item[(ii)] $X$ is a smooth projective surface containing no $\cF$-invariant $(-1)$-curve $E$ such that $K_{\cF}\cdot E \le 0$.
\end{itemize}
\end{definition}

\begin{definition}\label{def:logminfoliation}
A foliated surface $(X,\cF)$ is called \emph{log minimal} if
\begin{itemize}
\item[(i)] $\cF$ has at most log canonical singularities, and
\item[(ii)] $X$ is a smooth projective surface containing no $(-1)$-curve $E$ such that $K_{\cF}\cdot E \le \iota(E)$.
\end{itemize}
\end{definition}

\subsection{A generalization of $Z$-index}

Let $C$ be an irreducible $\cF$-invariant curve on a smooth surface $X$.
If $C$ is singular at a point $p$, the index ${\rm Z}(\cF,C,p)$ may be negative (cf.~\cite[p.~15]{Bru15}). 
To remedy this, we introduce a non-negative invariant $h_p(\cF,C)$.

Let $\nu = a(x,y)\tfrac{\partial}{\partial x} + b(x,y)\tfrac{\partial}{\partial y}$ be a local generator of $\cF$ at $p=(0,0)$, 
and let $B$ be an $\cF$-invariant analytic branch through $p$ with minimal Puiseux parametrization $\varphi(t)=(\varphi_x(t),\varphi_y(t))$.
We define the multiplicity of $\cF$ along $B$ at $p$ by
\begin{equation}\label{eq:mupdefi}
\mu_p(\cF,B) :=
\begin{cases}
\nu_0(a(\varphi_x(t),\varphi_y(t)))-\nu_0(\varphi_x(t))+1, & \varphi_x(t)\neq0,\\
\nu_0(b(\varphi_x(t),\varphi_y(t)))-\nu_0(\varphi_y(t))+1, & \varphi_y(t)\neq0,
\end{cases}
\end{equation}
where $\nu_0(h)$ denotes the order of the zero of $h$ at $t=0$.

Let $\sigma:X'\to X$ be the blow-up at $p$ with exceptional divisor $E$, let $\cF'=\sigma^*\cF$ and $\bar{B}$ be the strict transform of $B$, and set $p':=E\cap \bar{B}$. Then (see \cite[p.~291]{Car94})
\begin{equation}\label{equ:mu-blowup}
\mu_p(\cF,B) = \mu_{p'}(\cF',\bar{B}) + m_p(B)(l(p)-1).
\end{equation}

\begin{lemma}\label{lem:mu_p(F,C)>=0}
The following statements hold.
\begin{itemize}
\item[\rm(1)] $\mu_p(\cF,B)\ge0$, and $\mu_p(\cF,B)=0$ if and only if $p\notin {\rm Sing}(\cF)$.
\item[\rm(2)] $\mu_p(\cF,B)=1$ if and only if either $p$ has two non-zero eigenvalues or $p$ is a saddle-node with $B$ as its strong separatrix.  
In particular, if $m_p(B)\ge2$, then $p$ is a dicritical singularity with local generator 
$\nu=x\tfrac{\partial}{\partial x}+\lambda y\tfrac{\partial}{\partial y}$ for some $\lambda\in\bQ^+$.
\item[\rm(3)] If $\mu_p(\cF,B)\ge2$ and $p$ is canonical, then $m_p(B)=1$ and $p$ is a saddle-node with $B$ as its weak separatrix.
\end{itemize}
\end{lemma}

\begin{proof}
(1) It follows from \cite[p.~291]{Car94}. 

\noindent(2) Assume $\mu_p(\cF,B)=1$. If $m_p(B)=1$, then in suitable coordinates
\[
B=(y=0),\qquad 
\nu=(x^m w(x)+y\,u(x,y))\tfrac{\partial}{\partial x}+y\,v(x,y)\tfrac{\partial}{\partial y}, 
\quad w(0)\neq0.
\]
Taking the minimal parametrisation $\varphi(t)=(t,0)$ yields $\mu_p(\cF,B)=m$. 
Hence $\mu_p(\cF,B)=1$ if and only if $p$ has two non-zero eigenvalues or $p$ is a saddle-node whose strong separatrix is $B$.

Assume now $m_p(B)\ge2$. If $l(p)\ge2$, then \eqref{equ:mu-blowup} gives $\mu_p(\cF,B)\ge2$, a contradiction. 
Thus $l(p)=1$. We distinguish cases as in \cite[p.~2--8]{Bru15}.

\begin{itemize}
\item[\bf(a)] $p$ is a nilpotent singularity. 
In suitable coordinates,
\[
\nu=(y+u(x,y))\tfrac{\partial}{\partial x}+v(x,y)\tfrac{\partial}{\partial y},
\qquad u,v \text{ vanish to order }\ge2 .
\]
Let $\varphi(t)=(\varphi_x(t),\varphi_y(t))$ be a minimal parametrisation of $B$ at $p$, where $\varphi_x\varphi_y\neq0$.
By \eqref{eq:mupdefi}, $\mu_p(\cF,B)=1$ yields
\[
\nu_0(\varphi_y+u(\varphi_x,\varphi_y))=\nu_0(\varphi_x),\qquad
\nu_0(v(\varphi_x,\varphi_y))=\nu_0(\varphi_y),
\]
where $\nu_0$ denotes the vanishing order at $t=0$. 
Hence $\nu_0(\varphi_y)=\nu_0(\varphi_x)$ and $2\nu_0(\varphi_x)\le\nu_0(\varphi_y)$, which forces
$\nu_0(\varphi_x)=\nu_0(\varphi_y)=0$, a contradiction.

\item[\bf(b)] $p$ is a non-degenerate reduced singularity. Then $p$ has exactly two smooth separatrices, so $m_p(B)=1$, a contradiction.

\item[\bf(c)] $p$ is a saddle-node. Both the strong and weak separatrices (if it exists) are smooth, hence $m_p(B)=1$, a contradiction.

\item[\bf(d)] $p$ is dicritical with
$\nu=x\tfrac{\partial}{\partial x}+\lambda y\tfrac{\partial}{\partial y}$ for $\lambda\in\bQ^+$.

\item[\bf(e)] $p$ is a canonical but not reduced singularity with
$\nu=x\tfrac{\partial}{\partial x}+(ny+x^n)\tfrac{\partial}{\partial y}$,
where $n=\lambda$ or $1/\lambda$. Then $(x=0)$ is the unique separatrix at $p$, so $m_p(B)=1$, a contradiction.
\end{itemize}
This proves the “only if” part. The converse is immediate.
\smallskip

\noindent(3) 
Since $p$ is canonical, only cases (b), (c), and (e) need to be considered. 
As $\mu_p(\cF,B)\ge 2$, cases (b) and (e) are excluded. 
Moreover, $B$ cannot be a strong separatrix, so $p$ is a saddle-node with $B$ as its weak separatrix and $m_p(B)=1$.
\end{proof}

\begin{remark}\label{remk:CS}
In case (d), if $B=(x=0)$ or $(y=0)$, then ${\rm CS}(\cF,B,p)=1/\lambda$ or $\lambda$, both of which are positive.
If $B$ is a separatrix through $p$ but different from $(x=0)$ and $(y=0)$, then by considering the minimal resolution $\pi:(X',\cF')\to (X,\cF)$ of $p$ (cf.~Theorem~\ref{thm:Seidenberg}), we see that
$\bar{B}\cap \pi^{-1}(p)$ does not contain singularities of $\cF'$ (cf.~\cite[p.~7--8]{Bru15}), where $\bar{B}$ is the strict transform of $B$.
Let $p_1,\ldots,p_r$ denote the blow-up points of $\pi$.
Then by \cite[p.~30]{Bru15},
\[
{\rm CS}(\cF,B,p)
= \sum_{q\in\pi^{-1}(p)\cap \bar{B}}{\rm CS}(\cF',\bar{B},q)
+ \sum_{i=1}^r m_{p_i}(B)^2
= \sum_{i=1}^r m_{p_i}(B)^2 > 0.
\]

In case (e), by definition we have
${\rm CS}(\cF,B,p)={\rm CS}(\cF,x=0,p)=1/n>0$.
\end{remark}

\begin{definition}
Define
\begin{equation}\label{equ:defhp(F,C)}
h_p(\cF,C) := \sum_{B\in C(p)} \mu_p(\cF,B),
\end{equation}
where $C(p)$ is the set of branches of $C$ at $p$. 
\end{definition}

\begin{proposition}\label{prop:hp(F,C)}
We have $h_p(\cF,C)\ge0$, with the following properties:
\begin{itemize}
\item[(1)] $h_p(\cF,C)=0$ if and only if $p\notin {\rm Sing}(\cF)$; in this case, $C$ is smooth at $p$.
\item[(2)] $h_p(\cF,C)=1$ if and only if $C$ has a unique branch at $p$ and either $p$ has two non-zero eigenvalues or $p$ is a saddle-node with $C$ as its strong separatrix.  
In particular, if $m_p(C)\ge2$, then $p$ is a dicritical singularity with a local generator $v=x\tfrac{\partial}{\partial x}+\lambda y\tfrac{\partial}{\partial y}$ for $\lambda\in\bQ^+$ after a suitable choice of coordinates.
\item[(3)] If $h_p(\cF,C)=1$ and either ${\rm CS}(\cF,C,p)<0$ or $p$ is canonical, then $C$ is smooth at $p$.
Moreover, under this assumption, if ${\rm CS}(\cF,C,p)<0$, then $p$ is a non-degenerate reduced singularity of $\cF$.
\item[(4)] If $h_p(\cF,C)=2$ and $p$ is canonical, then one of the following cases occurs:
\begin{itemize}
\item[(i)]  $p$ is a node of $C$ and a non-degenerate reduced singularity of $\cF$.
\item[(ii)]  $p$ is a smooth point of $C$ and a saddle-node of $\cF$ with $C$ as its weak separatrix.
\end{itemize}
\end{itemize}
\end{proposition}

\begin{proof}
(1) and (2) follow from Lemma~\ref{lem:mu_p(F,C)>=0}.  
(3) follows from (2), since the case $m_p(C)\ge 2$ is excluded when either ${\rm CS}(\cF,C,p)<0$ or $p$ is canonical. 
The final statement follows from Lemma~\ref{lem:Z-CS-index-for-reduced} and Remark~\ref{remk:CS}.

For (4), write $C(p)=\{B_1,\dots,B_k\}$. Since $h_p(\cF,C)=2$,  $k\le 2$.  
If $k=1$, then $\mu_p(\cF,B_1)=2$, so by Lemma~\ref{lem:mu_p(F,C)>=0}(3), $C$ is the weak separatrix of a saddle-node. 
If $k=2$, then $p$ is a non-degenerate reduced singularity of $\cF$, hence a node of $C$.
\end{proof}

Motivated by the Cerveau--Lins Neto formula for foliations on $\bP^2$ (see \cite[p.~885]{CLN91}), we obtain the following generalization (see also \cite{LW24}). 
\begin{lemma} \label{lem:C-LNformula}
For an irreducible $\cF$-invariant curve $C$, we have
\begin{equation}\label{eq:C-LNformula}
2-2g(C)+K_{\cF}\cdot C=\sum\limits_{p\in C} h_p(\cF,C), 
\end{equation}
where $g(C)$ denotes the geometric genus of $C$. 
\end{lemma}
\begin{proof}
Let $\sigma:X'\to X$ be the blow-up at $p\in C$ with exceptional divisor $E$, and let $\bar{C}$ be the strict transform of $C$, meeting $E$ at points $q_1,\dots,q_r$. Denote by $\cF'$ the pullback foliation of $\cF$.
 By \eqref{equ:KF-blow-up}, \eqref{equ:mu-blowup} and \eqref{equ:defhp(F,C)},
\[
K_{\cF'}\cdot \bar{C} = K_{\cF}\cdot C + m_p(C)(1-l(p)),\qquad
h_p(\cF,C) = \sum_{i=1}^r h_{q_i}(\cF',\bar{C}) + m_p(C)(l(p)-1).
\]
It follows that
\[
\sum_{p\in C} h_p(\cF,C) - K_{\cF}\cdot C
=
\sum_{q\in \bar{C}} h_q(\cF',\bar{C}) - K_{\cF'}\cdot \bar{C},
\]
so the quantity
$\sum_{p\in C} h_p(\cF,C) - K_{\cF}\cdot C$
is invariant under blow-ups.
By Theorem~\ref{thm:Seidenberg}, we may assume that $\cF$ is reduced and $C$ is smooth. 
Then $p_a(C)=g(C)$ and $h_p(\cF,C) = Z(\cF,C,p)$ for every singularity $p$ of $\cF$ on $C$. 
By \eqref{equ:Z}, we obtain
\[
\sum_{p\in C} h_p(\cF,C) - K_{\cF}\cdot C = 2 - 2g(C).
\]
\end{proof}

\begin{remark}
If $C$ is smooth, then $h_p(\cF,C)={\rm Z}(\cF,C,p)$ for any $p\in C$, and \eqref{eq:C-LNformula} coincides with \eqref{equ:Z}. 
In contrast, for singular points, $h_p(\cF,C)\ge0$ while ${\rm Z}(\cF,C,p)$ may be negative (cf.~\cite[p.~15]{Bru15}).
\end{remark}

\begin{corollary}\label{cor:KFC=-1}
Assume that $C$ is an irreducible $\cF$-invariant curve with $C^2<0$. Then $K_{\cF}\cdot C\ge -1$.
Moreover, equality holds if and only if $C$ is a smooth rational $\cF$-invariant curve 
with exactly one non-degenerate reduced singularity of $\cF$.
\end{corollary}

\begin{proof}
By Proposition~\ref{prop:hp(F,C)} and Lemma~\ref{lem:C-LNformula}, 
\[
K_{\cF}\cdot C = \sum_{p\in C} h_p(\cF,C) - 2 + 2g(C) \ge -2 + 2g(C) \ge -2.
\]

If $K_{\cF}\cdot C=-2$, then $\sum_{p\in C} h_p(\cF,C)=0$, so $\cF$ has no singularities on $C$ by Proposition~\ref{prop:hp(F,C)}(1). 
Hence $C^2=0$ by \eqref{equ:CS}, contradicting $C^2<0$. Thus $K_{\cF}\cdot C\ge -1$.

If $K_{\cF}\cdot C=-1$, then $\sum_{p\in C} h_p(\cF,C)=1$ and $g(C)=0$, so $C$ has a unique singular point $p$. 
Since $C^2\le -1$, we have ${\rm CS}(\cF,C,p)=C^2\le -1$. 
By Proposition~\ref{prop:hp(F,C)}(3), it follows that $m_p(C)=1$ and that $p$ is a non-degenerate reduced singularity of $\cF$. 
In particular, $p_a(C)=g(C)=0$.

Conversely, if $C$ is a smooth rational $\cF$-invariant curve with exactly one non-degenerate reduced singularity $p$ of $\cF$, then $g(C)=p_a(C)=0$ and $h_p(\cF,C)=1$, 
so \eqref{eq:C-LNformula} gives $K_{\cF}\cdot C=-1$.
\end{proof}

\subsection{Technical tools for algebraic surfaces}

\begin{lemma}\label{lem:Hodgeindex}
Let $D_1$ and $D_2$ be divisors on a smooth projective surface $X$.
If $(\lambda D_1+\mu D_2)^2>0$ for some $\lambda,\mu\in\bR$, then
\[
D_1^2 D_2^2 \le (D_1\cdot D_2)^2,
\]
with equality if and only if a nonzero rational linear combination of $D_1$ and $D_2$ is numerically trivial.
(See \cite[Corollary~3.5]{Reid}.)
\end{lemma}

\begin{lemma}\label{lem:neg-def}
Let $D=\sum_{i=1}^n a_i C_i$ be a $\bQ$-divisor such that the intersection matrix $(C_i\cdot C_j)_{1\le i,j\le n}$ is negative definite.
\begin{itemize}
\item[(1)] If $D\cdot C_i\le 0$ for all $i$, then $D\ge 0$.
\item[(2)] If $E$ is an effective $\bQ$-divisor and $(E-D)\cdot C_j\le 0$ for all $j$, then $E-D\ge 0$.
\end{itemize}
(See \cite[Lemmas~14.9 and~14.15]{Luc01}.)
\end{lemma}
Although (2) follows from (1), we state it separately since this is the comparison form used later.

\section{\texorpdfstring{$(D,\cF)$}{(D,F)}-chains}

\subsection{\texorpdfstring{$\cF$}{F}-chains}
\begin{definition}[$\cF$-chain]\label{def:F-chains}
Let $\cF$ be a foliation on a surface $X$. A compact curve $\Theta\subset X$ is called an \emph{$\cF$-chain} if 
  \begin{itemize}
  \item[(1)] $\Theta$ is a Hirzebruch--Jung string, $\Theta=\Gamma_1+\cdots+\Gamma_r$;
  \item[(2)] each irreducible component $\Gamma_j$ is $\cF$-invariant;
  \item[(3)] ${\rm Sing}(\cF)\cap \Theta$ are reduced and non-degenerate;
  \item[(4)] ${\rm Z}(\cF,\Gamma_1)=1$ and  ${\rm Z}(\cF,\Gamma_i)=2$ for all $i\geq2$. 
  \end{itemize}
An $\cF$-chain is called \emph{maximal} if it cannot be contained in other $\cF$-chains.
(See \cite[Definition~8.1]{Bru15}.)
\end{definition}

Let $\Theta=\Gamma_1+\cdots+\Gamma_r$ be an $\cF$-chain. 
Then its irreducible components $\Gamma_i$ satisfy
\[
\Gamma_i^2=-e_i\le -2,\quad 
\Gamma_i\cdot \Gamma_{i+1}=1,\quad
\Gamma_i\cdot \Gamma_j=0 \ \text{for } j\ge i+2.
\]
It is associated with two sequences of integers (cf.~\cite[Ch.~III, Sec.~5]{BPV04}):
\begin{equation}\label{ineq:u_k>=k}
n=\lambda_0>\cdots>\lambda_{r+1}=0,\qquad 
0=\mu_0<\cdots<\mu_{r+1}=n,
\end{equation}
uniquely determined by $\mu_1=\lambda_r=1$ and
\begin{equation}\label{eq1}
\lambda_{i-1}-e_i\lambda_i+\lambda_{i+1}=0,\quad
\mu_{i-1}-e_i\mu_i+\mu_{i+1}=0 \quad (i=1,\dots,r).
\end{equation}
Moreover,
\[
\frac{n}{\lambda_1}
= e_1-\DF{1}{e_2-\DF{1}{\ddots-\DF{1}{e_r}}},
\qquad
\frac{n}{\mu_r}
= e_r-\DF{1}{e_{r-1}-\DF{1}{\ddots-\DF{1}{e_1}}}.
\]
By induction on $j$, $\lambda_i\mu_j-\lambda_j\mu_i$ is a positive multiple of $n$ for $i<j$. 
In particular, $\lambda_i\mu_{i+1}-\lambda_{i+1}\mu_i=n$.

\begin{remark}
Let $[e_k,\dots,e_l]$ denote the determinant of the intersection matrix $(-\Gamma_i\cdot\Gamma_j)_{k\le i,j\le l}$. 
Then $\mu_0=\lambda_{r+1}=0$, $\mu_1=\lambda_r=1$, and
\[
\lambda_{i-1}=[e_i,\dots,e_r],\quad 
\mu_{i+1}=[e_1,\dots,e_i],\quad (i=1,\dots,r).
\]
In particular, $n=\mu_{r+1}=\lambda_0=[e_1,\dots,e_r]$.
\end{remark}

\begin{lemma}\label{lem:CS=-uk+1uk}
Let $p_k=\Gamma_k\cap \Gamma_{k+1}$ for $k=1,\cdots, r-1$, and let $p_r$ denote  another singularity of $\cF$ on $\Gamma_r$. 
Then 
\begin{equation}\label{eq:CS=-uk+1uk}
{\rm CS}(\cF,\Gamma_k,p_k)=-\frac{\mu_{k+1}}{\mu_k},\quad k=1,\cdots,r.
\end{equation}
\end{lemma}
\begin{proof}
We argue by induction on $k$. 
For $k=1$, \eqref{equ:CS} gives 
\[
{\rm CS}(\cF,\Gamma_1,p_1)=\Gamma_1^2=-e_1=-\mu_2/\mu_1.
\] 
Assume ${\rm CS}(\cF,\Gamma_k,p_k)=-\mu_{k+1}/\mu_k$. Then
\begin{align*}
{\rm CS}(\cF,\Gamma_{k+1},p_{k+1})
&=\Gamma_{k+1}^2-{\rm CS}(\cF,\Gamma_{k+1},p_k) \quad\text{(by \eqref{equ:CS})}\\
&=\Gamma_{k+1}^2-\frac{1}{{\rm CS}(\cF,\Gamma_k,p_k)} \quad\text{(Lemma~\ref{lem:Z-CS-index-for-reduced}(1))}\\
&=-e_{k+1}+\frac{\mu_k}{\mu_{k+1}}
=-\frac{e_{k+1}\mu_{k+1}-\mu_k}{\mu_{k+1}}\\
&=-\frac{\mu_{k+2}}{\mu_{k+1}} \quad\text{(by \eqref{eq1})}.
\end{align*}
\end{proof}

Note that the intersection matrix $(\Gamma_i\cdot \Gamma_j)_{1\leq i,j\leq r}$ is negative definite. 
For each $i$, there exists a unique effective $\bQ$-divisor $M_i(\Theta)$ supported on $\Theta$ such that
$$M_i(\Theta)\cdot\Gamma_i=-1 \quad\text{and}\quad M_i(\Theta)\cdot\Gamma_j=0\quad \text{for $j\neq i$}.$$ 
  By a straightforward  computation, one has
  \begin{equation}\label{eq:M_i(Theta)}
     M_i(\Theta)=\frac{\lambda_i}{n}\sum_{k=1}^i\mu_k\Gamma_k+\frac{\mu_i}{n}\sum_{k=i+1}^r\lambda_k\Gamma_k. 
  \end{equation}  
  
   Let $A$ be a $\bQ$-divisor. We  define a $\bQ$-divisor 
   \begin{equation} 
    E(A,\Theta):=\sum\limits_{i=1}^r (A\cdot\Gamma_i) M_i(\Theta).
   \end{equation}
   Obviously, $E(A,\Theta)\cdot \Gamma_i=-A\cdot \Gamma_i$ for each $i$.  
   
\begin{lemma}
$E(A,\Theta)=\sum\limits_{i=1}^r\gamma_i\Gamma_i$, where 
\begin{equation}\label{eq:gamma_iECA}
  \gamma_i=\frac{\lambda_i}{n}\sum_{k=1}^i\mu_k A\cdot \Gamma_k+\frac{\mu_i}{n}\sum_{k=i+1}^r\lambda_k A\cdot\Gamma_k,\quad\text{for $i=1,\cdots,r$}.
  \end{equation}  
\end{lemma} 
\begin{proof}
  By the definition of $E(A,\Theta)$ and the equality \eqref{eq:M_i(Theta)}, it is clear.
\end{proof}

   \begin{lemma}\label{lem:EACgeq0}
      Assume that $A\cdot \Gamma_i\leq 0$ for   $i\geq2$.  Then $E(A,\Theta)\geq 0$ and $\gamma_r>0$ hold if and only if
  $$\sum_{k=1}^r\mu_kA\cdot\Gamma_k>0.$$
      In this case, $0<\gamma_i<A\cdot\Gamma_1$ for all $i$.
   \end{lemma}
   \begin{proof} 
   By the equality \eqref{eq:gamma_iECA}, we have
   $$\gamma_i=\frac{\lambda_i}{n}\sum\limits_{k=1}^r\mu_kA\cdot \Gamma_k-\sum\limits_{k=i+1}^r\frac{\lambda_i\mu_k-\lambda_k\mu_i}{n}A\cdot \Gamma_k. $$ 
   Since $\lambda_i\mu_k-\lambda_k\mu_i\geq n$ for $k>i$ and $A\cdot \Gamma_k\leq 0$ for $k\geq2$, we have
   \begin{equation}\label{ineq:gamma_i>=lambda_igamma_r}
  \gamma_i\geq \frac{\lambda_i}{n}\sum\limits_{k=1}^r\mu_k A\cdot\Gamma_k=\lambda_i\gamma_r.
   \end{equation}
   Hence $E(A,\Theta)\geq 0$ and $\gamma_r>0$ if and only if $\sum_{k=1}^r\mu_kA\cdot\Gamma_k>0$. 

   On the other hand, by the equality \eqref{eq:gamma_iECA} and our assumptions, we obtain
   \[
\gamma_i
\le \frac{\lambda_i}{n}\,\mu_1\, A\cdot\Gamma_1
< A\cdot\Gamma_1,
\qquad i=1,\dots,r.
\]
\end{proof}

\subsection{\texorpdfstring{$(D,\cF)$}{(D,F)}-chains}
Let $\Theta=\Gamma_1+\cdots+\Gamma_r$ be an $\cF$-chain and let $D$ be a $\bQ$-divisor. 
In the case $A:=-(K_\cF+D)$, we write $E(A,\Theta)$ as $M(D,\Theta)$.

\begin{definition}[$(D,\cF)$-chain]\label{def:DFchain}  
  An $\cF$-chain $\Theta$ is said to be a \emph{$(D,\cF)$-chain} if $M(D,\Theta)=\sum_{i=1}^r\gamma_i\Gamma_i$ is effective with $\gamma_r>0$.  
  Such a $(D,\cF)$-chain  is called  \emph{maximal} if it cannot be contained in  other  $(D,\cF)$-chains. 
  \end{definition} 
  
  We denote $M(\Theta):=M(0,\Theta)$ for convenience when $D=0$. 
In this case, each $\cF$-chain $\Theta$ is a $(0,\cF)$-chain. 
When $D=0$, the equality \eqref{eq:gamma_iECA} yields the following:
  \begin{lemma}
  $M(\Theta)=\sum\limits_{i=1}^r\gamma_i\Gamma_i$, where
  \begin{equation}
  \gamma_i=\frac{\lambda_i}{n},\quad i=1,\cdots,r.
  \end{equation}
\end{lemma}
\medskip

\noindent{\bf Assumption.} In the following, we assume that $D\cdot \Gamma\geq 0$ for every component $\Gamma$ of a $(D,\cF)$-chain. 
This holds, for example, when $D=aK_X$ with $a\geq 0$, or when $D$ is an effective divisor whose components are not $\cF$-invariant.
  
\begin{proposition}\label{prop:FchainisDF}
Let $\Theta=\Gamma_1+\dots+\Gamma_r$ be an $\cF$-chain.  
Then  $\Theta$ is  a $(D,\cF)$-chain if and only if
$$\sum_{k=1}^r\mu_kD\cdot \Gamma_k<1.$$ 
In particular, $D\cdot \Gamma_k<\tfrac{1}{\mu_k}\leq\tfrac{1}{k}$   for $k=1,\dots,r$.
\end{proposition}
\begin{proof}
By Lemma~\ref{lem:EACgeq0} and the definition of a $(D,\cF)$-chain, $\Theta$ is a $(D,\cF)$-chain if and only if
\[
-\sum_{k=1}^r \mu_k (K_{\cF}+D)\cdot \Gamma_k > 0.
\]
Since $K_{\cF}\cdot \Gamma_1=-1$ and $K_{\cF}\cdot \Gamma_k=0$ for $k\ge 2$, this is equivalent to
\[
\sum_{k=1}^r \mu_k D\cdot \Gamma_k < \mu_1 = 1.
\]
Under the standing assumption that $D\cdot \Gamma_k \ge 0$ for all $k$, we then have
$D\cdot \Gamma_k < \tfrac{1}{\mu_k} \le \tfrac{1}{k}$, 
where the second inequality follows from \eqref{ineq:u_k>=k}.
\end{proof}
   
\begin{example}
The following examples of $(D,\cF)$-chains follow from Proposition~\ref{prop:FchainisDF}.
\begin{itemize}
\item[(1)] If $C$ is an irreducible non-$\cF$-invariant curve, then a $(C,\cF)$-chain is an $\cF$-chain disjoint from $C$.

\item[(2)] If $D=\sum_{i=1}^l a_i C_i$, where $a_i\in[\tfrac{1}{2},1]$ and each $C_i$ is not $\cF$-invariant,
then a $(D,\cF)$-chain is an $\cF$-chain $\Theta=\Gamma_1+\cdots+\Gamma_r$ satisfying
\[
D\cdot \Gamma_1<1,\qquad D\cdot \Gamma_i=0 \quad \text{for } i\geq 2.
\]

\item[(3)] If $D=K_X$, then a $(D,\cF)$-chain is an $\cF$-chain $\Theta=\Gamma_1+\cdots+\Gamma_r$ such that each $\Gamma_i$ is a $(-2)$-curve.
\end{itemize}
\end{example}

\begin{proposition}\label{prop:DF-subchain}
Let $\Theta=\Gamma_1+\dots+\Gamma_r$ be a $(D,\cF)$-chain. 
Then:
\begin{itemize}
\item [(1)]  $\lfloor M(D,\Theta)\rfloor=0$. More precisely,  $0<\lambda_i\gamma_r\leq \gamma_i\leq \frac{\lambda_i}{n}$ and 
\begin{equation}\label{eq:gamma_r}
\gamma_r=\frac{1}{n}\left(1-\sum_{i=1}^{r}\mu_i D\cdot\Gamma_{i}\right).
\end{equation}
In particular, $n\gamma_r\,M(\Theta)\leq M(D,\Theta)\leq M(\Theta).$

\item[(2)] Let $\Gamma$ be an irreducible $\cF$-invariant curve not contained in $\Theta$. 
Then $\Theta+\Gamma$ is a $(D,\cF)$-chain if and only if it is an $\cF$-chain and
$D\cdot \Gamma<M(D,\Theta)\cdot \Gamma$.
  
\item[(3)] Each sub-chain $\Theta_t:=\Gamma_1+\dots+\Gamma_t$ ($t\leq r$) is a $(D,\cF)$-chain, and  
\[
0<M(D,\Theta_1)<M(D,\Theta_2)<\dots <M(D,\Theta_{r-1})<M(D,\Theta).
\] 
\end{itemize}   
\end{proposition} 
\begin{proof}
(1) follows from a direct computation using \eqref{eq:gamma_iECA} and \eqref{ineq:gamma_i>=lambda_igamma_r}.
\smallskip

\noindent(2) By Proposition~\ref{prop:FchainisDF}, an $\cF$-chain $\Theta+\Gamma$ is a $(D,\cF)$-chain if and only if 
\[
\sum_{k=1}^r \mu_k\, D\cdot \Gamma_k + \mu_{r+1} D\cdot \Gamma < 1.
\]
Since $\mu_{r+1}=n$ (see \eqref{ineq:u_k>=k}), this is equivalent to
\[
n D\cdot \Gamma < 1 - \sum_{k=1}^r \mu_k D\cdot \Gamma_k.
\]
By \eqref{eq:gamma_r}, the right-hand side equals $n\gamma_r$, and hence
$D\cdot \Gamma < \gamma_r = M(D,\Theta)\cdot \Gamma$.
\smallskip

\noindent(3) For $t\le r$, the chain $\Theta_t$ has the same sequence $\mu_0<\mu_1<\cdots<\mu_t$ satisfying \eqref{eq1}. 
Thus $\sum_{i=1}^t \mu_i D\cdot \Gamma_i<1$, and $\Theta_t$ is a $(D,\cF)$-chain by Proposition~\ref{prop:FchainisDF}.

Moreover,
\[
(M(D,\Theta_{t+1})-M(D,\Theta_t))\cdot \Gamma_i=
\begin{cases}
0,& i\le t,\\
D\cdot \Gamma_{t+1}-M(D,\Theta_t)\cdot \Gamma_{t+1},& i=t+1.
\end{cases}
\]
By (2), we have $D\cdot \Gamma_{t+1}-M(D,\Theta_t)\cdot \Gamma_{t+1}<0$.
Hence $M(D,\Theta_{t+1})-M(D,\Theta_t)>0$ by Lemma~\ref{lem:neg-def}.
\end{proof}

\begin{corollary}\label{coro:eKXF-chainisF-chain}
Let $\Theta=\Gamma_1+\cdots+\Gamma_r$ be an $\cF$-chain and $\epsilon\in[0,\frac{1}{n}]$, 
where $n=[e_1,\cdots,e_r]$ is defined as above. 
Then $\Theta$ is an $(\epsilon K_X,\cF)$-chain.
\end{corollary}

\begin{proof}
Since $\epsilon K_X\cdot \Gamma_i=\epsilon(-2-\Gamma_i^2)\leq \frac{1}{n}(e_i-2)$, 
Proposition~\ref{prop:FchainisDF} reduces the proof to showing that
\[
1-\frac{1}{n}\sum_{i=1}^{r}\mu_i(e_i-2)>0.
\]
Since $n=\mu_{r+1}$, it suffices to prove
\[
N_k:=\sum_{i=1}^{k}\mu_i(e_i-2)<\mu_{k+1},\qquad k=1,\cdots,r.
\]

We argue by induction on $k$.
For $k=1$,  $N_1=\mu_1(e_1-2)=e_1-2<e_1=\mu_2$.
Assume $N_k<\mu_{k+1}$. Then
\begin{align*}
N_{k+1}&=N_k+\mu_{k+1}(e_{k+1}-2)\\
&<\mu_{k+1}+\mu_{k+1}(e_{k+1}-2)=(e_{k+1}\mu_{k+1}-\mu_k)-(\mu_{k+1}-\mu_k)\\
&<\mu_{k+2},
\end{align*}
where we use the fact that  $\mu_{k+2}=e_{k+1}\mu_{k+1}-\mu_k$ and $\mu_{k+1}>\mu_k$.
\end{proof}

\subsection{Curves outside \texorpdfstring{$(D,\cF)$}{(D,F)}-chains}
Let $C$ be an irreducible curve that is not contained in any $(D,\cF)$-chain 
(for example, when $C$ is not $\cF$-invariant).

Let $\Theta = \Gamma_1 + \cdots + \Gamma_r$ be a $(D,\cF)$-chain, and define
\begin{equation}\label{eqdef:E(C,Theta)}
E(C,\Theta) = \sum_{i=1}^r \beta_i \Gamma_i
\end{equation}
to be the effective $\bQ$-divisor satisfying $E(C,\Theta)\cdot \Gamma_i = -C \cdot \Gamma_i$ for all $i$, as introduced earlier.

  Let $x\geq0$. By Proposition \ref{prop:FchainisDF}, there exists a $(D+xC,\cF)$-chain contained in $\Theta$  if and only if $(D+x C)\cdot \Gamma_1<1$. 
If this condition holds, let $\Theta'=\Gamma_1+\dots+\Gamma_t$ be the maximal $(D+xC,\cF)$-chain contained in $\Theta$. 
Otherwise, set $\Theta'=0$ and $M(D+xC,\Theta')=0$. 

\begin{lemma}\label{lem:xE+V'-V>=0}
Under the notation above, we have
\begin{equation}
  x E(C,\Theta)+M(D+xC,\Theta')-M(D,\Theta)\geq0. 
\end{equation}
\end{lemma}
\begin{proof}
Set $Z:=x E(C,\Theta)+M(D+x C,\Theta')-M(D,\Theta)$.
By Lemma~\ref{lem:neg-def}, it suffices to show that $Z\cdot\Gamma_i\le 0$ for all $i$.

If $\Theta'=\Theta$, then $Z\cdot \Gamma_i=0$ for all $i$, hence $Z=0$. 
If $\Theta'=0$, then $(D+xC)\cdot \Gamma_1 \ge 1$ by Proposition~\ref{prop:FchainisDF}, and
$Z=x E(C,\Theta)-M(D,\Theta)$.
Thus,
\[
Z\cdot \Gamma_i=
\begin{cases}
1-(D+xC)\cdot \Gamma_1,&\text{if } i=1,\\
-(D+xC)\cdot\Gamma_i,&\text{if } i\ge 2.
\end{cases}
\]
Hence $Z\cdot \Gamma_i \le 0$ for all $i$.

Now assume $\Theta'\ne 0,\Theta$, i.e., $1\le t<r$. 
A direct computation shows that
\[
Z\cdot \Gamma_i=
\begin{cases}
0,&\text{if } i\le t,\\
-(D+xC)\cdot\Gamma_{t+1}+M(D+xC,\Theta')\cdot\Gamma_{t+1},&\text{if } i=t+1,\\
-(D+xC)\cdot\Gamma_i,&\text{if } i\ge t+2.
\end{cases}
\]
By Proposition~\ref{prop:DF-subchain}(2), $Z\cdot \Gamma_{t+1}\leq0$.
Thus $Z\cdot\Gamma_i\le 0$ for all $i$.
\end{proof}

\begin{corollary}\label{coro:theta(C,T)}
Assume that $C$ meets a $(D,\cF)$-chain 
$\Theta = \sum_{i=1}^r \Gamma_i$ 
but is not contained in $\Theta$. 
Then there exists a minimal $x_0 \ge 0$ such that, for all $x \ge x_0$, the maximal $(D + x C, \cF)$-chain $\Theta'$ contained in $\Theta$ is disjoint from $C$. 
We denote this number $x_0$ by $\theta(C,\Theta)$.
In particular, we have $0 < \theta(C,\Theta) \le 1$.
\end{corollary}
\begin{proof}
Let $\Theta_t=\Gamma_1+\cdots +\Gamma_t$ ($0\leq t<r$) be a sub-chain such that 
$C\cdot \Gamma_k=0$ for all $1\leq k\leq t$, and $C\cdot \Gamma_{t+1}>0$, 
where we set $\Theta_0=0$. 
By Proposition~\ref{prop:FchainisDF}, the above condition is equivalent to the existence of $x\geq0$ such that
\[
\sum_{k=1}^t \mu_k(D+xC)\cdot \Gamma_k < 1
\quad \text{and} \quad
\sum_{k=1}^{t+1} \mu_k(D+xC)\cdot \Gamma_k \geq 1.
\]
Thus,
\[
x \geq \theta(C,\Theta)
:= \frac{1}{\mu_{t+1}  C\cdot\Gamma_{t+1}}
\left(1 - \sum_{k=1}^{t+1} \mu_k D\cdot\Gamma_k\right).
\]
Note that $\theta(C,\Theta)$ depends only on $C$, $\Theta$, and $D$.
By Proposition~\ref{prop:FchainisDF}, we have
$0 < \theta(C,\Theta) \leq \tfrac{1}{\mu_{t+1}} \leq 1$.
\end{proof} 

For convenience, we define $\theta(C,\Theta) := 0$ whenever $\Theta$ is disjoint from $C$.

\begin{lemma}\label{lem:theta<alpha(C)}
Let $\Theta = \Gamma_1 + \cdots + \Gamma_r$ be a $(D, \cF)$-chain, and let $C$ be an irreducible non-$\cF$-invariant curve.
\begin{itemize}
\item[\rm(1)] If $D=\sum_{i=1}^l a_i C_i$ is an effective $\bQ$-divisor with $a_i \in [0,1]$ for all $i$, and each $C_i$ is not $\cF$-invariant, then 
\[
\theta(C,\Theta) \le \alpha_C, \quad \text{where } 
\alpha_C = 
\begin{cases}
1 - a_i,& \text{if } C = C_i \text{ for some } i,\\
1,& \text{otherwise}.
\end{cases}
\]
\item[\rm(2)] If $D = \epsilon K_X$ with $\epsilon \in [0,\tfrac{1}{2})$ and $\Gamma_1^2 \le -3$, then 
$\theta(C,\Theta) \le 1 - \epsilon$.
\end{itemize}
\end{lemma}
\begin{proof}
This follows from the definition of $\theta(C,\Theta)$ and Proposition~\ref{prop:FchainisDF}.
\end{proof}

Next, we introduce an index of $C$ that depends only on $C$ and the $(D,\cF)$-chains.

\begin{definition}\label{def:DFindexofC}
The \emph{$(D,\cF)$-index} of $C$ is defined as
\begin{equation}
\theta(C) := \max_{\Theta} \bigl\{ \theta(C,\Theta) \bigr\},
\end{equation}
where the maximum is taken over all maximal $(D,\cF)$-chains $\Theta$.
\end{definition}

By Corollary \ref{coro:theta(C,T)}, we have $0 \le \theta(C) \le 1$. 
Moreover, $\theta(C) = 0$ if and only if $C$ is disjoint from all maximal $(D,\cF)$-chains.

\medskip

We now define
\begin{equation}\label{eq:defcE(C)}
\mathcal{E}(C) := C + \sum_{\Theta} E(C,\Theta), 
\quad 
W_C := K_{\cF} + D - \sum_{\Theta} M(D,\Theta),
\end{equation}
where the sum runs over all maximal $(D,\cF)$-chains $\Theta$ intersecting $C$. 
Clearly, for any component $\Gamma$ of each $\Theta$, we have
$\mathcal{E}(C) \cdot \Gamma = W_C \cdot \Gamma = 0$. 

\begin{lemma}\label{lem:(W+cC)cC>=0}
Let $C$ be an irreducible non-$\cF$-invariant curve such that 
\[
(K_{\cF}+D)\cdot C + x  C^2 \ge 0
\]
for some $x \ge \theta(C)$. Then
\begin{equation}\label{eq:((W+xcC)cC>=0)}
(W_C + x  \mathcal{E}(C)) \cdot \mathcal{E}(C) \ge 0.
\end{equation}
\end{lemma}

\begin{proof}
Let $\Theta = \Gamma_1 + \cdots + \Gamma_r$ be a maximal $(D,\cF)$-chain intersecting $C$, and let 
$\Theta' = \Gamma_1 + \cdots + \Gamma_t$ be the maximal $(D+xC,\cF)$-chain contained in $\Theta$, 
setting $\Theta'=0$ if $(D+xC)\cdot \Gamma_1 \ge 1$.  
Since $x \ge \theta(C)$, we have $C \cdot \Theta' = 0$, which implies
\[
C \cdot \bigl(x E(C,\Theta) - M(D,\Theta)\bigr)
=
C \cdot \bigl(x E(C,\Theta) + M(D+xC,\Theta') - M(D,\Theta)\bigr) \ge 0,
\]
where the last inequality follows from Lemma \ref{lem:xE+V'-V>=0}.  
Therefore, we obtain
\begin{align*}
&(W_C + x \mathcal{E}(C)) \cdot \mathcal{E}(C)
= (W_C + x  \mathcal{E}(C)) \cdot C \\
=& (K_{\cF} + D)\cdot C + x  C^2 
   + \sum_{\Theta} C \cdot \bigl(x E(C,\Theta) - M(D,\Theta)\bigr)
  \ge 0.
\end{align*}
\end{proof}

\begin{remark}
In the case $D=0$, the assumption of Lemma~\ref{lem:(W+cC)cC>=0} is automatically satisfied by taking $x=1$. 
Thus, Lemma~\ref{lem:(W+cC)cC>=0} recovers \cite[Proposition~2.13]{CF18}, although our proof is different.
\end{remark}

\begin{lemma}\label{lem:CSqi<0}
Let $C$ be an irreducible $\cF$-invariant curve not contained in any $(D,\cF)$-chain, and let $q_1,\dots,q_h$ be the singularities of $\cF$ on $C$ outside the nodes of $\mathcal{E}(C)$.
Then    
\begin{equation}
\cE(C)^2=\sum_{i=1 }^h{\rm CS}(\cF,C,q_i). 
\end{equation}
Furthermore, if the intersection matrix of  all components of $\mathcal{E}(C)$ is negative definite, then the following hold:
\begin{itemize}
\item[(1)] $\sum_{i=1 }^h{\rm CS}(\cF,C,q_i)<0$.
\item[(2)] If $\sum_{i=1}^h h_{q_i}(\cF,C)=1$, then $h=1$ and $q_1$ is a non-degenerate reduced singularity of $\cF$,
and so $C$ is smooth.
(Here $h_{q_i}(\cF,C)$ is defined in \eqref{equ:defhp(F,C)}.)
\end{itemize}
\end{lemma}

\begin{proof} 
Let $\Theta=\Gamma_1+\cdots+\Gamma_r$ be a maximal $(D,\cF)$-chain intersecting $C$ at the point $p=C\cap\Gamma_r$.
By Lemma~\ref{lem:CS=-uk+1uk}, we have
\[
{\rm CS}(\cF,C,p)=\frac{1}{{\rm CS}(\cF,\Gamma_r,p)}=-\frac{\mu_r}{n}.
\]
On the other hand, by \eqref{eq:gamma_iECA}, 
\[
E(C,\Theta)\cdot C=\gamma_r=\frac{\mu_r}{n},
\]
where $E(C,\Theta)=\sum_{i=1}^r\gamma_i\Gamma_i$. Therefore,
\begin{equation}\label{equ:E(C,T)=-CS}
E(C,\Theta)\cdot C=-{\rm CS}(\cF,C,p).
\end{equation}

Now let $\Theta_1,\dots,\Theta_k$ be the maximal $(D,\cF)$-chains intersecting $C$, and set $p_i=\Theta_i\cap C$. 
By the Camacho--Sad formula (cf.~\eqref{equ:CS}) and \eqref{equ:E(C,T)=-CS}, we obtain
\[
\mathcal{E}(C)^2
= C\cdot \mathcal{E}(C)
= C^2 - \sum_{i=1}^k {\rm CS}(\cF,C,p_i)
= \sum_{i=1}^h {\rm CS}(\cF,C,q_i).
\]

For the second part, (1) follows immediately from the above identity, and (2) follows from (1) together with Proposition~\ref{prop:hp(F,C)}.
\end{proof} 

\section{Zariski decomposition of \texorpdfstring{$K_{\cF}+D$}{KF+D}}

\begin{definition}\label{def:D(Delta,epsilon)}
We define the $\bQ$-divisor $D_{\Delta,\epsilon}$ as:
\begin{equation}\label{eq:defD}
D_{\Delta,\epsilon}:=\Delta+\epsilon K_X,
\end{equation}
where  $\epsilon\in[0,1]\cap\bQ$, and  $\Delta = \sum_{i=1}^l a_i C_i$ is an effective $\mathbb{Q}$-divisor such that
\begin{itemize}
\item    $a_i \in [0,1]$ if $C_i$ is not $\cF$-invariant, and 
\item $\Delta\cdot\Gamma\geq0,$ for any irreducible $\cF$-invariant curve $\Gamma$.
\end{itemize}
\end{definition} 
In particular, for any irreducible $\cF$-invariant curve $\Gamma$ with $\Gamma^2<0$, we have
$D_{\Delta,\epsilon}\cdot \Gamma \ge -\epsilon$.
Moreover, if $D_{\Delta,\epsilon}\cdot \Gamma<0$, then $\epsilon>0$ and $\Gamma$ is a $(-1)$-curve.

\begin{theorem}\label{mainthm:<1/4}
Let $\cF$ be a foliation on a smooth projective surface $X$, and let $D=D_{\Delta,\epsilon}$ as in Definition~\ref{def:D(Delta,epsilon)} with $\epsilon\in[0,\frac{1}{4}]$. Assume that there are no $(-1)$-curves of the following types:
\begin{itemize}
\item[(1)] $E$ is a non-$\cF$-invariant $(-1)$-curve with $K_{\cF}\cdot E=1$ and $\epsilon K_X\cdot E<0$;
\item[(2)] $\Gamma$ is an $\cF$-invariant $(-1)$-curve such that either
\begin{itemize}
    \item[(2a)] $D\cdot \Gamma \ge 0$, $K_{\cF}\cdot \Gamma \le 0$, and $K_{\cF}\cdot \Gamma + 2 D\cdot \Gamma \le 1$, or
    \item[(2b)] $-\frac{1}{4} \le D\cdot \Gamma <0$ and $K_{\cF}\cdot \Gamma \le 0$.
\end{itemize}
\end{itemize}

If $K_{\cF}+D$ is pseudo-effective with Zariski decomposition $K_{\cF}+D=P(D)+N(D)$, then
\[
N(D)=\sum_{i=1}^s M(D,\Theta_i),
\]
where $\{\Theta_1,\dots,\Theta_s\}$ is the set of all maximal $(D,\cF)$-chains on $X$. In particular, $\lfloor N(D)\rfloor=0$.
\end{theorem}

\begin{remark}
\begin{itemize}
\item[(1)] The range of $\epsilon$ in Theorem~\ref{mainthm:<1/4} controls the singularities of $\cF$: one may assume $\cF$ is canonical when $\epsilon=0$, and log canonical when $0<\epsilon\le \tfrac{1}{4}$.

\item[(2)] In \cite{LLTX}, the authors consider the case $\epsilon = 0$, with $\Delta_{\rm red}$ non-$\mathcal{F}$-invariant and $\tfrac{1}{2} \le \mathrm{Coeff}(\Delta) < 1$.

\item[(3)] With minor modifications, our method also applies when $D$ is an effective $\bQ$-divisor with $D \wedge \mathbf{B}_-(K_{\cF}+D)=0$ and $\cF$ is canonical on a smooth surface without $\cF$-exceptional curves, as in \cite[Sec.~5]{JhL}.
\end{itemize}
\end{remark}

By relaxing the assumptions on the singularities of $\cF$, we obtain the following generalization of \cite[Theorem~1.2]{LW24}, which treats the case $\Delta=0$ and $\epsilon=1$.

\begin{theorem}\label{mainthm:1/4<e<1}
Let $\cF$ be a foliation on a smooth projective surface $X$, and let $D=D_{\Delta,\epsilon}$ as in Definition~\ref{def:D(Delta,epsilon)} with $\epsilon\in[0,1]$. 
Assume that there are no $(-1)$-curves of the following types:
\begin{itemize}
\item[(1)] $E$ is a non-$\cF$-invariant $(-1)$-curve with $K_{\cF}\cdot E=1$ and $\epsilon K_X\cdot E<0$;
\item[(2)] $\Gamma$ is an $\cF$-invariant $(-1)$-curve such that either
\begin{itemize}
    \item[(2a)] $D\cdot \Gamma \ge 0$, $K_{\cF}\cdot \Gamma \le 0$, and $K_{\cF}\cdot \Gamma + 2 D\cdot \Gamma \le 1$, or
    \item[(2b)] $-1 \le D\cdot \Gamma <0$ and $K_{\cF}\cdot \Gamma \le 1$.
\end{itemize}
\end{itemize}

If $K_{\cF}+D$ is pseudo-effective with Zariski decomposition $K_{\cF}+D=P(D)+N(D)$, then
\[
N(D)=\sum_{i=1}^s M(D,\Theta_i),
\]
where $\{\Theta_1,\dots,\Theta_s\}$ is the set of all maximal $(D,\cF)$-chains on $X$. In particular, $\lfloor N(D)\rfloor=0$.
\end{theorem}
\begin{proof}
The proof is similar to that of Theorem~\ref{mainthm:<1/4}.
\end{proof}

\subsection{Proof of Theorem \ref{mainthm:<1/4}}

Let $\Theta_1,\dots,\Theta_s$ be all maximal $(D,\cF)$-chains. 
By the separatrix theorem (cf.~Theorem~\ref{thm:separatrix}), these chains are pairwise disjoint.

Set
\[
V := \sum_{i=1}^s M(D,\Theta_i), \qquad W := K_{\cF}+D-V.
\]
By definition, $W\cdot \Gamma = 0$ for any component $\Gamma$ of $V$.

\begin{lemma}\label{lem:N(D)geqV}
We have $N(D)\geq V$. 
In particular, $W$ is pseudo-effective and its Zariski decomposition has negative part $N(D)-V$.
\end{lemma}
\begin{proof}
For any component $\Gamma$ of $V$, 
we have $(N(D)-V)\cdot \Gamma \le W\cdot \Gamma = 0$.
Since the intersection matrix of components of $V$ is negative definite, Lemma~\ref{lem:neg-def}(2) implies $N(D)-V \ge 0$.
The remaining assertion follows immediately.
\end{proof}

\begin{lemma}\label{lem:Dtang-index}
If $C$ is an irreducible non-$\cF$-invariant curve with $C^2<0$, then
\[
(K_{\cF}+D)\cdot C + \theta(C) C^2 \ge 0,
\]
where $\theta(C)$ is defined in Definition~\ref{def:DFindexofC}. 
\end{lemma}

\begin{proof}
Suppose $\epsilon=0$. Then $D=\Delta=\sum_{i=1}^l a_i C_i$. So
\[
(K_{\cF}+D)\cdot C + \alpha_C C^2
\ge K_{\cF}\cdot C + C^2 \ge 0,
\]
where 
\[
\alpha_C :=
\begin{cases}
1-a_i, & \text{if } C=C_i \text{ for some } i,\\
1, & \text{otherwise}.
\end{cases}
\]

Suppose $\epsilon>0$. If $C$ is not a $(-1)$-curve, then $K_X\cdot C \ge 0$, which implies
\[
(K_{\cF}+D)\cdot C + \alpha_C C^2
\ge (K_{\cF}+\Delta)\cdot C + \alpha_C C^2 \ge 0.
\]
If $C$ is a $(-1)$-curve, then $C^2=-1$ and $K_X\cdot C=-1$, which implies $K_{\cF}\cdot C \ge 2$ by condition~(1). So
\[
(K_{\cF}+D)\cdot C + \alpha_C C^2
\ge K_{\cF}\cdot C + C^2 + \epsilon K_X\cdot C
\ge 1 - \epsilon \ge 0.
\]

Finally, since $C^2<0$ and $\alpha_C \ge \theta(C)$ (by Proposition~\ref{prop:FchainisDF} and Definition~\ref{def:DFindexofC}), we conclude that
$(K_{\cF}+D)\cdot C + \theta(C) C^2 \ge 0$.
\end{proof}

\begin{corollary}\label{cor:(KF+D)C>0-C-nonFinv}
Under the assumption and notation of Lemma \ref{lem:Dtang-index}, 
if $\lfloor\Delta\rfloor=0$,
then $(K_{\cF}+D)\cdot C>0$.
\end{corollary}
\begin{proof}
Since $\lfloor \Delta \rfloor = 0$, we have $\alpha_C > 0$.
By the proof of Lemma~\ref{lem:Dtang-index}, we obtain
$(K_{\cF}+D)\cdot C \ge -\alpha_C C^2 > 0$.
\end{proof}

\begin{lemma}\label{lem:forNON-FinvWC>=0}
For any irreducible non-$\cF$-invariant curve $C$, we have $W\cdot C \ge 0$.
\end{lemma}
\begin{proof}
Suppose $W\cdot C<0$. Then $W\cdot \cE(C)=W\cdot C<0$, where $\cE(C)$ is defined in \eqref{eq:defcE(C)}.

Since $W$ is pseudo-effective, it follows that $C$ is contained in the negative part $N_W$ of $W$.  
By Lemma~\ref{lem:N(D)geqV}, we have $N_W = N(D)-V$, so $C$ is a component of $N(D)$.  
Moreover, all other components of $\cE(C)$ lie in $V$, and hence ${\rm Supp}\,\cE(C)\subseteq {\rm Supp}\,N(D)$, whose intersection matrix is negative definite.  
Thus $\cE(C)^2<0$.

By Lemmas~\ref{lem:(W+cC)cC>=0} and~\ref{lem:Dtang-index}, we have
\[
\big(W+\theta(C)\cE(C)\big)\cdot \cE(C)\ge 0.
\]
Since $\theta(C)\ge 0$, this implies $W\cdot \cE(C)\ge 0$, a contradiction.  
Hence $W\cdot C\ge 0$.
\end{proof}

\begin{proof}[Proof of Theorem \ref{mainthm:<1/4}] 
By the definition of the Zariski decomposition, it suffices to show that $W$ is nef and that $W\cdot \Gamma=0$ for any component $\Gamma$ of $V$. 
The latter follows from the definition of  $(D,\cF)$-chains, so it remains to prove $W\cdot C\ge0$ for every irreducible curve $C$.
If $C$ is not $\cF$-invariant, then $W\cdot C\geq0$ by Lemma~\ref{lem:forNON-FinvWC>=0}.
Hence we may assume that $C$ is $\cF$-invariant.
    
Suppose $C$ is an irreducible $\cF$-invariant curve such that $W\cdot C<0$. 
Then $C$ lies in $N(D)$ but is not contained in $V$.
Hence the intersection matrix of the components of $C+V$ is negative definite, and in particular $C^2<0$.

Assume that $C$ meets the maximal $(D,\cF)$-chains $\Theta_1,\dots,\Theta_k$ transversely at $p_1,\dots,p_k$, and contains $h$ other singularities $q_1,\dots,q_h$ of $\cF$. 
Write $M(D,\Theta_i)=\sum_{j=1}^{r(i)}\gamma_{ij}\Gamma_{ij}$. 
By Proposition~\ref{prop:DF-subchain},  $\gamma_{i,r(i)}\leq \frac{1}{n_i}$, where $n_i=[e_{i1},\cdots,e_{i,r(i)}]$ with $e_{ij}=-\Gamma_{ij}^2$.
Reordering if necessary, assume $2\leq n_1\leq n_2\leq\cdots\leq n_k$.
Then
\begin{equation}\label{ineq:VC}
    V\cdot C=\sum_{i=1}^k\gamma_{i,r(i)}\leq\sum_{i=1}^k\frac{1}{n_i}\leq\frac{k}{2}.
\end{equation}
By Lemmas~\ref{lem:C-LNformula},
    \begin{equation}\label{eq:KFC}
    K_{\cF}\cdot C=2g(C)-2+k+h(\cF,C),
    \end{equation}
    where $h(\cF,C)=\sum_{i=1}^h h_{q_i}(\cF,C)$ and $h_{q_i}(\cF,C)\geq1$ is defined in \eqref{equ:defhp(F,C)}.
Hence,
\begin{equation}\label{ineq:WC}
    0>W\cdot C =K_{\cF}\cdot C-V\cdot C+D\cdot C\geq 2g(C)-2+ h (\cF,C)+\frac{k}{2}+D\cdot C. 
\end{equation}

\begin{itemize}
\item[\rm Claim.] We have $1 \le h(\cF,C) \le 2$ and $g(C)=0$. 
If $h(\cF,C)=1$, then $h=1$, $q_1$ is a non-degenerate reduced singularity of $\cF$, and $C\cong\bP^1$.

Since $D\cdot C\geq -\frac{1}{4}$, it follows from \eqref{ineq:WC} that $h(\cF,C)\leq 2$.  
If $h(\cF,C)=0$, then $h=0$, $C$ is smooth and $C+\Theta_1+\dots+\Theta_k$ forms a tree, which is excluded by the separatrix theorem (cf.~Theorem~\ref{thm:separatrix}).  
Thus $h(\cF,C)\in\{1,2\}$, and $g(C)=0$ by \eqref{ineq:WC}.  
The statement for $h(\cF,C)=1$ follows from Lemma~\ref{lem:CSqi<0}.
\end{itemize}
\medskip

\noindent{\bf Case 1.} $D\cdot C\geq0$. Then \eqref{ineq:WC} implies  that 
$h(\cF,C)=1$ and $k\leq 1$, so $C\cong\bP^1$ and $K_{\cF}\cdot C=k-1\in\{-1,0\}$. 
    
If $k=0$, then $K_{\cF}\cdot C=-1$ and $0\le D\cdot C<1$.  
Proposition~\ref{prop:FchainisDF} gives two possibilities: either $C$ is the first component of a $(D,\cF)$-chain (contradicting $C\not\subset V$), or $C$ is an $\cF$-invariant $(-1)$-curve with $K_{\cF}\cdot C=-1$ and $0\le D\cdot C<1$ (impossible by condition (2a)).

If $k=1$, then $K_{\cF}\cdot C=0$ and 
\[
0>W\cdot C = D\cdot C - M(D,\Theta_1)\cdot C.
\] 
Proposition~\ref{prop:DF-subchain}(2) gives two possibilities: either $\Theta_1+C$ is a $(D,\cF)$-chain (contradicting $C\not\subset V$),
or $C$ is an $\cF$-invariant $(-1)$-curve with $K_{\cF}\cdot C=0$ and $0\le D\cdot C < M(D,\Theta_1)\cdot C \le 1/2$ (impossible by condition (2a)).

\medskip

\noindent{\bf Case 2.} $-\frac{1}{4}\leq D\cdot C<0$.
Then $C$ is an $\cF$-invariant $(-1)$-curve.  
If $K_{\cF}\cdot C\le0$, this is impossible by condition~(2b).  
Assume $K_{\cF}\cdot C\ge1$, so $h(\cF,C)+k\ge3$. On the other hand, \eqref{ineq:WC} gives $h(\cF,C)+k/2<2-D\cdot C\le 9/4$.  
By the claim, $h(\cF,C)\in\{1,2\}$, so the only possibility is
\[
h(\cF,C)=1,\quad k=2,\quad K_{\cF}\cdot C=1 .
\]

Suppose, for contradiction, that this occurs. Then \eqref{ineq:VC} and \eqref{ineq:WC} imply
\[
\frac{1}{n_1}+\frac{1}{n_2}\geq V\cdot C>K_{\cF}\cdot C+D\cdot C\geq\frac{3}{4},
\]
so $(n_1,n_2)=(2,2)$ or $(2,3)$. 

If $(n_1,n_2)=(2,2)$, then $\Theta_i = \Gamma_{i1}$ with $\Gamma_{i1}^2=-2$ for $i=1,2$.  
Then the intersection matrix of $\Gamma_{11}+\Gamma_{21}+C$ is not negative definite, a contradiction.

If  $(n_1,n_2)=(2,3)$, there are two subcases:
\begin{itemize}
\item[\rm (A)] $\Theta_1=\Gamma_{11}$ with $\Gamma_{11}^2=-2$ and $\Theta_2=\Gamma_{21}+\Gamma_{22}$ with $\Gamma_{21}^2=\Gamma_{22}^2=-2$.
\item[\rm (B)] $\Theta_1=\Gamma_{11}$ with $\Gamma_{11}^2=-2$ and $\Theta_2=\Gamma_{21}$ with $\Gamma_{21}^2=-3$.
\end{itemize}
Case (A) is impossible as before.
For case (B), we have $\gamma_{1,r(1)}=\gamma_{11}\leq\frac{1}{2}$ and
$$\gamma_{2,r(2)}=\gamma_{21}=\frac{(K_{\cF}+D)\cdot \Gamma_{21}}{\Gamma_{21}^2}=\frac{1}{3}(1-D\cdot\Gamma_{21})\leq\frac{1}{3}(1-\epsilon),$$
so
$$V\cdot C=\gamma_{1,r(1)}+\gamma_{2,r(2)}\leq \frac{5}{6}-\frac{\epsilon}{3}.$$
Then
$$W\cdot C=K_{\cF}\cdot C+D\cdot C-V\cdot C\geq 1-\epsilon-(\frac{5}{6}-\frac{\epsilon}{3})=\frac{1-4\epsilon}{6}\geq0,$$
where we use the assumption $\epsilon\leq\frac{1}{4}$.
This leads to a contradiction.
\end{proof}

\subsection{Other Variants}
In Theorem \ref{mainthm:<1/4}, the singularities of $\cF$ are generally not canonical when $\epsilon > 0$.
However, if  $K_{\cF}+\Delta$ is big and $0<\epsilon\ll1$, then the singularities of $\cF$ can be canonical. 
More precisely, we have the following result:

\begin{theorem}\label{thm:eKXcanonical}
Let $(X,\cF)$ be a minimal foliated surface (cf.~Definition~\ref{def:minfoliation}).  
Let $D=D_{\Delta,\epsilon}=\Delta+\epsilon K_X$ as in Definition~\ref{def:D(Delta,epsilon)} with $\lfloor\Delta\rfloor=0$ and
$0\leq \epsilon \leq \tfrac{1}{3\cdot i(\Delta,\cF)}$.

If $K_{\cF} + \Delta$ is big, then $K_{\cF} + D$ is pseudo-effective and admits a Zariski decomposition
$K_{\cF} + D = P(D) + N(D)$,
where
\[
N(D)=\sum_{i=1}^s M(D,\Theta_i),
\]
and $\{\Theta_1, \dots, \Theta_s\}$ is the set of all maximal $(D,\cF)$-chains on $X$.

Moreover, if $0\leq\epsilon<\frac{1}{3\cdot i(\Delta,\cF)}$, then $K_{\cF}+D$ is big.
\end{theorem}
\begin{proof}
Let
\[
W=K_{\cF}+D-V, \qquad V=\sum_{i=1}^s M(D,\Theta_i).
\]
Recall that $K_{\cF}+\Delta=P(\Delta)+N(\Delta)$.
By the definitions of $(\Delta,\cF)$-chains and $(D,\cF)$-chains and Lemma~\ref{lem:neg-def},  $N(\Delta)-V\ge 0$.
Thus we can write
\[
W=P(\Delta)+\epsilon K_X+\big(N(\Delta)-V\big).
\]

\begin{itemize}
\item[\rm Claim.] $3i(\Delta,\cF) P(\Delta)+K_X$ is nef. Then $P(\Delta)+\epsilon K_X$ is nef for any
$0\leq\epsilon\leq\tfrac{1}{3\cdot i(\Delta,\cF)}$.
(In the case $\Delta=0$, this coincides with \cite[Lemma~6.5]{PS19}.)

\smallskip

Let $A=i(\Delta,\cF)P(\Delta)$, which is an integral divisor.
By the cone theorem for surfaces, it suffices to check non-negativity on all $K_X$-negative extremal curves $C$.
Let $C$ be such a curve. Then $K_X\cdot C\ge -3$ by \cite[Theorem~1.24]{KM98}.

 If $P(\Delta)\cdot C>0$, then $A\cdot C\ge 1$, and hence
$(3A+K_X)\cdot C \ge 0$.

 If $P(\Delta)\cdot C=0$, then $C$ is $\cF$-invariant by Lemma~\ref{lem:FnoninvCsuchthatP(Delta)C=0}(2) together with $\lfloor\Delta\rfloor=0$.
Applying Lemma~\ref{lem:P(D)C=0-C-Finv} to $\Delta$ (i.e.\ $\epsilon=0$), case (F) is excluded.
Hence $C$ is either a component of a $(\Delta,\cF)$-chain or falls into one of the cases (A)--(E), so that either $C^2\le -2$ or $p_a(C)=1$.
In either case, $K_X\cdot C = 2p_a(C)-2-C^2 \ge 0$,
and thus $(3A+K_X)\cdot C\ge 0$.
\end{itemize}

Let $C$ be an irreducible curve on $X$.
If $C$ is a component of $V$, then $W\cdot C=0$.
If $C$ is not a component of $N(\Delta)$, then
\[
W\cdot C=(P(\Delta)+\epsilon K_X)\cdot C+(N(\Delta)-V)\cdot C\ge 0.
\]

Next assume that $C$ is a component of $N(\Delta)$ but not contained in $V$.
Then $P(\Delta)\cdot C=0$ and there exists a maximal $(\Delta,\cF)$-chain $\Theta=\Gamma_1+\cdots+\Gamma_r$ such that
\begin{itemize}
\item[\rm(i)] $\Theta'=\Gamma_1+\cdots+\Gamma_t$ is a maximal $(D,\cF)$-chain, where $0\le t<r$ (with $t=0$ meaning $\Theta'=0$);
\item[\rm(ii)] $C=\Gamma_k$ for some $k\ge t+1$ (hence $C^2\le -2$ and $K_X\cdot C\ge 0$).
\end{itemize}
By Theorem~\ref{thm:separatrix}, $C$ is disjoint from all other $(\Delta,\cF)$-chains.

If $t=0$, then $\Theta'=0$, which implies $V\cdot C=0$ and $D\cdot\Gamma_1\ge 1$ by Proposition~\ref{prop:FchainisDF}.
\begin{itemize}
\item If $C=\Gamma_1$, then
$(N(\Delta)-V)\cdot C=N(\Delta)\cdot C=(K_{\cF}+\Delta)\cdot C=-1+\Delta\cdot C$.
Hence
\[
W\cdot C= \epsilon K_X\cdot C+(-1+\Delta\cdot C)=D\cdot C-1\ge 0.
\]

\item If $C=\Gamma_k$ for $k\ge 2$, then
$(N(\Delta)-V)\cdot C=N(\Delta)\cdot C=\Delta\cdot C\ge 0$,
and so $W\cdot C\ge 0$.
\end{itemize}

If $t\ge 1$ and $k\ge t+2$, then
$(N(\Delta)-V)\cdot C=N(\Delta)\cdot C=\Delta\cdot C\ge 0$,
so $W\cdot C\ge 0$.

If $t\ge 1$ and $k=t+1$, then $N(\Delta)\cdot C=\Delta\cdot C$ and 
\[
V\cdot C=M(D,\Theta')\cdot C\le D\cdot C,
\]
where the last inequality follows from Proposition~\ref{prop:DF-subchain}(2).
Thus,
\[
W\cdot C=\epsilon K_X\cdot C+N(\Delta)\cdot C-V\cdot C
=(\epsilon K_X+\Delta)\cdot C-V\cdot C
=D\cdot C-V\cdot C\ge 0.
\]

Therefore, $K_{\cF}+D$ is pseudo-effective with the positive part $P(D)=W$ and the first statement holds.
\medskip

Now assume that $0\le \epsilon<\frac{1}{3i}$, where $i:=i(\Delta,\cF)$.
Then
\begin{align*}
P(D)
&=P(\Delta)+\epsilon K_X+(N(\Delta)-N(D))\\
&=(1-3i\epsilon)P(\Delta)+\epsilon(3iP(\Delta)+K_X)+(N(\Delta)-N(D))\\
&=(\text{big and nef})+(\text{nef})+(\text{effective}).
\end{align*}
Hence $P(D)$ is big by \cite[Proposition 2.61]{KM98}.
\end{proof}

\begin{corollary}\label{cor:eKXcanonical}
Let $\cF$ be a canonical foliation on a smooth surface $X$ without $\cF$-invariant $(-1)$-curves $E$ such that $K_{\cF}\cdot E\le 0$.
Assume that $K_{\cF}$ is big with Zariski decomposition $K_{\cF}=P+N$, and let
$\epsilon\in\left(0,\tfrac{1}{3\cdot i(\cF)}\right)\cap \bQ$.
Then:
\begin{itemize}
\item[(1)] $K_{\cF}+\epsilon K_X$ is big and admits a Zariski decomposition
$K_{\cF}+\epsilon K_X = P(\epsilon)+N(\epsilon)$,
where
\[
N(\epsilon)=\sum_{i=1}^s M(\epsilon K_X,\Theta_i),
\]
and $\{\Theta_1,\dots,\Theta_s\}$ is the set of all maximal $\cF$-chains on $X$.

\item[(2)] The null locus ${\rm Null}(P(\epsilon))$ consists of the following two types of curves:
\begin{itemize}
\item[(i)] components of $\cF$-chains;
\item[(ii)] $(-2)$-curves $C$ with $P\cdot C=0$ that are not contained in any $\cF$-chain.
\end{itemize}

\item[(3)]
$(1-3\epsilon i(\cF))^2 P^2 \le P(\epsilon)^2 \le (1+\gamma)^2 P^2$,
where
\[
\gamma=\max\left\{\frac{2\epsilon P\cdot K_X}{P^2}+1,\,0\right\}.
\]
In particular, $P(\epsilon)^2$ is bounded in terms of $P^2$, $i(\cF)$, $\epsilon$, and $P\cdot K_X$.
\end{itemize}
\end{corollary}
\begin{proof}
(1) follows directly from Theorem~\ref{thm:eKXcanonical}, 
noting that any $\cF$-chain is also a $(\epsilon K_X,\cF)$-chain by Corollary~\ref{coro:eKXF-chainisF-chain}.

\medskip

\noindent(2) Let $C$ be an irreducible curve with $P(\epsilon)\cdot C=0$.
If $C$ is a component of an $\cF$-chain, there is nothing to prove.
Assume that $C$ is not contained in any $\cF$-chain.  

Recall that
$P(\epsilon)=P+\epsilon K_X+(N-N(\epsilon))$,
so
\begin{equation}\label{equ:P(e)C=0}
0 = P(\epsilon)\cdot C 
= P\cdot C + \epsilon K_X\cdot C + (N-N(\epsilon))\cdot C.
\end{equation}

If $P\cdot C>0$, then by the definition of $i(\cF)$,
$P\cdot C \ge \tfrac{1}{i(\cF)} > 3\epsilon$.
Since $K_X\cdot C \ge -1$ and $(N-N(\epsilon))\cdot C \ge 0$, we get
$P(\epsilon)\cdot C \ge 3\epsilon - \epsilon = 2\epsilon > 0$,
a contradiction. Hence $P\cdot C=0$.

By \cite[Theorem~1.III.3.2]{MMcQ08}, either $C^2 \le -2$ or $p_a(C)=1$, and in particular $K_X\cdot C \ge 0$. 
Since also $(N-N(\epsilon))\cdot C \ge 0$, all terms in the equality~\eqref{equ:P(e)C=0} must vanish. 
In particular, $C$ is a $(-2)$-curve such that $P\cdot C=0$.
\smallskip

Conversely, let $C$ be a $(-2)$-curve with $P\cdot C=0$ which is not contained in any $\cF$-chain.
If $C\cap N=\emptyset$, then clearly $P(\epsilon)\cdot C=(N-N(\epsilon))\cdot C=0$.

Assume $C\cap N\neq\emptyset$. By \cite[Theorem~1.III.3.2]{MMcQ08} or \cite[Theorem~2.16]{CF18}, 
there exist exactly two $\cF$-chains $\Theta_1$ and $\Theta_2$ intersecting $C$.  
Each $\Theta_i$ consists of a single irreducible $\cF$-invariant $(-2)$-curve $\Gamma_i$, and
$M(\Theta_i)=M(\epsilon K_X,\Theta_i)$ for $i=1,2$,
where $M(\Theta_i):=M(0,\Theta_i)$. Hence
\[
P(\epsilon)\cdot C
= (N-N(\epsilon))\cdot C
= \sum_{i=1}^2 \bigl(M(\Theta_i)-M(\epsilon K_X,\Theta_i)\bigr)\cdot C
=0.
\]

\noindent(3)
Set $D_1:=P$ and $D_2:=\epsilon K_X + N - N(\epsilon)$. 
Then $D_1$ and $D_1+D_2$ are big and nef, $D_1\cdot D_2 = \epsilon P\cdot K_X$, and $D_1^2 = P^2$. 
Moreover, 
\[
D_1 + \frac{1}{3\epsilon i(\cF)}D_2
= P+\frac{1}{3i(\cF)}K_X+\frac{1}{3\epsilon i(\cF)}(N-N(\epsilon))
\]
is pseudo-effective, and $\gamma D_1 - D_2$ is pseudo-effective by Lemma~\ref{lem:HA}. 
Therefore,
\begin{align*}
P(\epsilon)^2
&={\rm Vol}(D_1+D_2)
={\rm Vol}\!\left((1-3\epsilon i(\cF))D_1
+3\epsilon i(\cF)\left(D_1+\tfrac{1}{3\epsilon i(\cF)}D_2\right)\right)\\
&\ge {\rm Vol}\big((1-3\epsilon i(\cF))D_1\big)
=(1-3\epsilon i(\cF))^2P^2,\\
P(\epsilon)^2
&={\rm Vol}(D_1+D_2)
={\rm Vol}\big((1+\gamma)D_1-(\gamma D_1-D_2)\big)\\
&\le {\rm Vol}\big((1+\gamma)D_1\big)
=(1+\gamma)^2P^2.
\end{align*}
\end{proof}

\begin{lemma}\label{lem:HA}
  Let $D_1$ be a nef and big $\bQ$-divisor on a smooth projective surface $X$.
Let $D_2$ be another $\bR$-divisor such that $D_1+D_2$ is nef and big. Then the $\bR$-divisor $\beta D_1-D_2$ is pseudo-effective, where 
$$\beta=\frac{2D_1\cdot D_2}{D_1^2}+1.$$
\end{lemma}
\begin{proof}
This follows from \cite[Lemma 3.5]{HA21}.
\end{proof}

\section{Canonical model of  \texorpdfstring{$K_{\cF}+D$}{KF+D}}

Let $\left(X, \cF, D = D_{\Delta, \epsilon}\right)$ be as in the setting of Theorem \ref{mainthm:<1/4}.
Consider the following $(K_{\cF}+D)$-non-positive birational contraction:
\begin{equation}
\sigma:(X,\cF, D)\to (Y,\cG,D_Y),
\end{equation}
where $K_{\cG}=\sigma_*K_{\cF}$, $D_Y=\sigma_*D$ and for any irreducible curve $C$ over $Y$, 
\[
(K_{\cG} + D_Y) \cdot C \le 0 \;\Longrightarrow\; C^2 \ge 0.
\]
We call $(Y,\cG,D_Y)$ the \emph{canonical model} of the triple $(X,\cF,D)$.

Let ${\rm Exc}(\sigma)$ denote the set of $\sigma$-exceptional curves over $X$. 
It is clear that:
\begin{itemize}
\item For any $C\in{\rm Exc}(\sigma)$, we have $P(D)\cdot C=0$. 
\item The intersection matrix of components of  ${\rm Exc}(\sigma)$ is negative definite.
\item If $K_{\cF}+D$ is big, then ${\rm Exc}(\sigma)={\rm Null}\,P(D)$.
\end{itemize}

In this section, we will discuss the structure of ${\rm Exc}(\sigma)$. 
The goal is to prove the following result:

\begin{theorem}\label{thm:Null}
Let $(X,\cF,D=D_{\Delta,\epsilon})$ be as in the setting of Theorem~\ref{mainthm:<1/4}, 
with the additional assumptions that $\lfloor\Delta\rfloor=0$, $\epsilon\in[0,\frac{1}{4})$, 
and that $\cF$ has at most canonical singularities when $\epsilon=0$.

Let $Z$ be a connected component of ${\rm Exc}(\sigma)$. 
Then every irreducible component of $Z$ is $\cF$-invariant, and $Z$ falls into one of the following cases:

\begin{itemize}
\item[(1)]$Z=\Gamma_1+\cdots+\Gamma_r$ is an $\cF$-chain, and one of the following holds:
\begin{itemize}
\item[\rm(1-1)] $Z$ is a maximal $(D,\cF)$-chain;
\item[\rm(1-2)] $D\cdot \Gamma_1=1$ and $D\cdot \Gamma_i=0$ for all $i\ge 2$;
\item[\rm(1-3)] there exists $t\le r$ such that $\Theta=\Gamma_1+\cdots+\Gamma_t$ is a maximal $(D,\cF)$-chain, 
$D\cdot\Gamma_{t+1}=M(D,\Theta)\cdot\Gamma_{t+1}\ (\le \tfrac{1}{2})$, and $D\cdot\Gamma_i=0$ for all $i\ge t+2$.
\end{itemize}
(See Figure~\ref{fig:Fchain}.)

\begin{figure}[htpt]
\centering
\def\svgwidth{\columnwidth}
\scalebox{0.4}{
\begingroup%
  \makeatletter%
  \providecommand\color[2][]{%
    \errmessage{(Inkscape) Color is used for the text in Inkscape, but the package 'color.sty' is not loaded}%
    \renewcommand\color[2][]{}%
  }%
  \providecommand\transparent[1]{%
    \errmessage{(Inkscape) Transparency is used (non-zero) for the text in Inkscape, but the package 'transparent.sty' is not loaded}%
    \renewcommand\transparent[1]{}%
  }%
  \providecommand\rotatebox[2]{#2}%
  \newcommand*\fsize{\dimexpr\f@size pt\relax}%
  \newcommand*\lineheight[1]{\fontsize{\fsize}{#1\fsize}\selectfont}%
  \ifx\svgwidth\undefined%
    \setlength{\unitlength}{126.14618764bp}%
    \ifx\svgscale\undefined%
      \relax%
    \else%
      \setlength{\unitlength}{\unitlength * \real{\svgscale}}%
    \fi%
  \else%
    \setlength{\unitlength}{\svgwidth}%
  \fi%
  \global\let\svgwidth\undefined%
  \global\let\svgscale\undefined%
  \makeatother%
  \begin{picture}(1,0.12199074)%
    \lineheight{1}%
    \setlength\tabcolsep{0pt}%
    \put(0.56121918,0.0263005){\color[rgb]{0,0,0}\makebox(0,0)[lt]{\lineheight{1.25}\smash{\begin{tabular}[t]{l}$\cdots\cdots$\end{tabular}}}}%
    \put(0,0){\includegraphics[width=\unitlength,page=1]{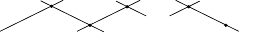}}%
    \put(0.25219959,0.07537052){\color[rgb]{0,0,0}\makebox(0,0)[lt]{\lineheight{1.25}\smash{\begin{tabular}[t]{l}$\Gamma_2$\end{tabular}}}}%
    \put(0.06765842,0.01258831){\color[rgb]{0,0,0}\makebox(0,0)[lt]{\lineheight{1.25}\smash{\begin{tabular}[t]{l}$\Gamma_1$\end{tabular}}}}%
    \put(0.78870066,0.0677606){\color[rgb]{0,0,0}\makebox(0,0)[lt]{\lineheight{1.25}\smash{\begin{tabular}[t]{l}$\Gamma_r$\end{tabular}}}}%
  \end{picture}%
\endgroup%
}
\caption{$\cF$-chains.}
\label{fig:Fchain}
\end{figure}

\item[(2)] $Z=\Gamma_1+\cdots+\Gamma_r$ is a chain of smooth rational $\cF$-invariant  curves such that
${\rm Z}(\cF,\Gamma_i)=2$ and $D\cdot\Gamma_i=0$, for all $i$.

\item[(3)] $Z=\Gamma_1+\cdots+\Gamma_r$ is a tree of smooth rational $\cF$-invariant curves such that
\begin{itemize}
\item[\rm(i)] $D\cdot \Gamma_i=0$ for all $i$;
\item[\rm(ii)] $\Gamma_3+\cdots+\Gamma_r$ is a chain with ${\rm Z}(\cF,\Gamma_3)=3$ and ${\rm Z}(\cF,\Gamma_i)=2$ for all $i\ge 4$, in particular, $\Gamma_i^2\le -2$ for all $i\ge 3$;
\item[\rm(iii)]  $\Gamma_1$ and $\Gamma_2$ are maximal $\cF$-chains with $\Gamma_1^2=\Gamma_2^2=-2$, each meeting $\Gamma_3$ transversely at distinct points, and disjoint from $\Gamma_j$ for all $j\ge 4$.
(See Figure~\ref{fig:Dihedral}.)
\end{itemize}

\begin{figure}[htpt]
\centering
\def\svgwidth{\columnwidth}
\scalebox{0.5}{
\begingroup%
  \makeatletter%
  \providecommand\color[2][]{%
    \errmessage{(Inkscape) Color is used for the text in Inkscape, but the package 'color.sty' is not loaded}%
    \renewcommand\color[2][]{}%
  }%
  \providecommand\transparent[1]{%
    \errmessage{(Inkscape) Transparency is used (non-zero) for the text in Inkscape, but the package 'transparent.sty' is not loaded}%
    \renewcommand\transparent[1]{}%
  }%
  \providecommand\rotatebox[2]{#2}%
  \newcommand*\fsize{\dimexpr\f@size pt\relax}%
  \newcommand*\lineheight[1]{\fontsize{\fsize}{#1\fsize}\selectfont}%
  \ifx\svgwidth\undefined%
    \setlength{\unitlength}{145.34549677bp}%
    \ifx\svgscale\undefined%
      \relax%
    \else%
      \setlength{\unitlength}{\unitlength * \real{\svgscale}}%
    \fi%
  \else%
    \setlength{\unitlength}{\svgwidth}%
  \fi%
  \global\let\svgwidth\undefined%
  \global\let\svgscale\undefined%
  \makeatother%
  \begin{picture}(1,0.1190861)%
    \lineheight{1}%
    \setlength\tabcolsep{0pt}%
    \put(0.10743431,0.09554869){\color[rgb]{0,0,0}\makebox(0,0)[lt]{\lineheight{1.25}\smash{\begin{tabular}[t]{l}$\Gamma_1(-2)$\end{tabular}}}}%
    \put(0.61917962,0.02282636){\color[rgb]{0,0,0}\makebox(0,0)[lt]{\lineheight{1.25}\smash{\begin{tabular}[t]{l}$\cdots\cdots$\end{tabular}}}}%
    \put(0,0){\includegraphics[width=\unitlength,page=1]{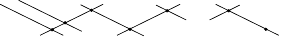}}%
    \put(-0.00072131,0.04560056){\color[rgb]{0,0,0}\makebox(0,0)[lt]{\lineheight{1.25}\smash{\begin{tabular}[t]{l}$\Gamma_2(-2)$\end{tabular}}}}%
    \put(0.35097975,0.0654145){\color[rgb]{0,0,0}\makebox(0,0)[lt]{\lineheight{1.25}\smash{\begin{tabular}[t]{l}$\Gamma_4$\end{tabular}}}}%
    \put(0.25521135,0.04229807){\color[rgb]{0,0,0}\makebox(0,0)[lt]{\lineheight{1.25}\smash{\begin{tabular}[t]{l}$\Gamma_3$\end{tabular}}}}%
    \put(0.8166121,0.05880981){\color[rgb]{0,0,0}\makebox(0,0)[lt]{\lineheight{1.25}\smash{\begin{tabular}[t]{l}$\Gamma_r$\end{tabular}}}}%
  \end{picture}%
\endgroup%
}
\caption{$D\cdot\Gamma_i=0$ for all $i$, and $\Gamma_j^2\le -2$ for $j\ge 3$.}
\label{fig:Dihedral}
\end{figure}

\item[(4)] $Z$ is an elliptic Gorenstein leaf with $D\cdot Z=0$, and one of the following holds:
\begin{itemize}
\item[\rm(4-1)] $Z$ is a rational nodal $\cF$-invariant curve such that $D\cdot Z=0$ and ${\rm Sing}(\cF)\cap Z$ coincides with the singular locus of $Z$;
\item[\rm(4-2)] $Z=\Gamma_1+\cdots+\Gamma_r$ is a cycle of smooth rational $\cF$-invariant curves such that $\Gamma_i^2\le -2$, $D\cdot\Gamma_i=0$ for all $i$, and ${\rm Sing}(\cF)\cap Z$ coincides with the singular locus of $Z$.
\end{itemize}
This case occurs only when $\epsilon=0$.
(See Figure~\ref{fig:egl}.)

\begin{figure}[htpt]
\centering
\def\svgwidth{\columnwidth}
\scalebox{0.4}{
\begingroup%
  \makeatletter%
  \providecommand\color[2][]{%
    \errmessage{(Inkscape) Color is used for the text in Inkscape, but the package 'color.sty' is not loaded}%
    \renewcommand\color[2][]{}%
  }%
  \providecommand\transparent[1]{%
    \errmessage{(Inkscape) Transparency is used (non-zero) for the text in Inkscape, but the package 'transparent.sty' is not loaded}%
    \renewcommand\transparent[1]{}%
  }%
  \providecommand\rotatebox[2]{#2}%
  \newcommand*\fsize{\dimexpr\f@size pt\relax}%
  \newcommand*\lineheight[1]{\fontsize{\fsize}{#1\fsize}\selectfont}%
  \ifx\svgwidth\undefined%
    \setlength{\unitlength}{119.720694bp}%
    \ifx\svgscale\undefined%
      \relax%
    \else%
      \setlength{\unitlength}{\unitlength * \real{\svgscale}}%
    \fi%
  \else%
    \setlength{\unitlength}{\svgwidth}%
  \fi%
  \global\let\svgwidth\undefined%
  \global\let\svgscale\undefined%
  \makeatother%
  \begin{picture}(1,0.35223707)%
    \lineheight{1}%
    \setlength\tabcolsep{0pt}%
    \put(0,0){\includegraphics[width=\unitlength,page=1]{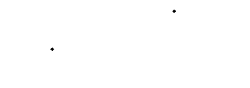}}%
    \put(0.40738535,0.17106308){\color[rgb]{0,0,0}\makebox(0,0)[lt]{\lineheight{1.25}\smash{\begin{tabular}[t]{l}or\end{tabular}}}}%
    \put(0,0){\includegraphics[width=\unitlength,page=2]{egl.pdf}}%
    \put(0.85259636,0.17131562){\color[rgb]{0,0,0}\makebox(0,0)[lt]{\lineheight{1.25}\smash{\begin{tabular}[t]{l}$\vdots$\end{tabular}}}}%
  \end{picture}%
\endgroup%
}
\caption{An elliptic Gorenstein leaf $Z$ with $D\cdot\Gamma=0$ for every irreducible component $\Gamma$ of $Z$.}
\label{fig:egl}
\end{figure}
\end{itemize}

In particular, let $p$ be the contraction point of $Z$. 
Then $p$ is a rational quotient singularity if $Z$ is of type {\rm(1)}--{\rm(3)}, 
and an elliptic Gorenstein singularity if $Z$ is of type {\rm(4)}.
\end{theorem}

\begin{remark}
Theorem~\ref{thm:Null} can be viewed as a generalization of \cite[Theorem~1.III.3.2]{MMcQ08} and \cite[Theorem~2.16]{CF18}, 
which treat the special case where $D=0$ and $(X,\cF)$ is a minimal foliated surface (cf.~Definition~\ref{def:minfoliation}).
\end{remark}
      
\begin{corollary}\label{coro:null}
Let $\left(X, \cF, D = D_{\Delta, \epsilon}\right)$ be as in the setting of Theorem \ref{thm:Null}.
Let $(Y,\cG,D_Y)$ be the canonical model of $(X,\cF,D)$. 
\begin{itemize}
\item[\rm(1)] If $\epsilon=0$, then $\cG$ has at most canonical singularities and $Y$ has at most rational quotient singularities, or elliptic Gorenstein singularities whose minimal resolutions have zero intersection with $D$.
\item[\rm(2)] If $\epsilon>0$ and $\cF$ has at most canonical (resp.\ log canonical) singularities, then $\cG$ has at most canonical (resp.\ log canonical) singularities and  $Y$ has at most rational quotient singularities.
\end{itemize}
\end{corollary} 

In the special case \(D=\epsilon K_X\) with \(\epsilon>0\), every component in the
dihedral configuration appearing in Theorem~5.1(3) is a \((-2)\)-curve. Indeed,
Theorem~5.1(3) gives \(D\cdot\Gamma_i=0\) for every such component \(\Gamma_i\).
Thus \(K_X\cdot\Gamma_i=0\). Since \(\Gamma_i\) is a smooth rational curve,
adjunction gives \(\Gamma_i^2=-2\).

Next, we show that the condition $\epsilon<\tfrac{1}{4}$ in Theorem~\ref{thm:Null} is necessary. 
The following example illustrates that $\cG$ may fail to be log canonical when $\epsilon=\tfrac{1}{4}$.

\begin{example}\label{ex:e=1/4}
Let $(X,\cF)$ be a foliated surface, where $\cF$ is locally defined by 
\[
\omega=x\,{\rm d}x+y^2\,{\rm d}y.
\]
Consider the minimal resolution $\sigma:(X',\cF')\to(X,\cF)$ over $p=(0,0)$ (see Figure~\ref{fig:1/4KX}).

\begin{figure}[htpt]
  \centering
  \def\svgwidth{\columnwidth}
  \scalebox{0.4}{
\begingroup%
  \makeatletter%
  \providecommand\color[2][]{%
    \errmessage{(Inkscape) Color is used for the text in Inkscape, but the package 'color.sty' is not loaded}%
    \renewcommand\color[2][]{}%
  }%
  \providecommand\transparent[1]{%
    \errmessage{(Inkscape) Transparency is used (non-zero) for the text in Inkscape, but the package 'transparent.sty' is not loaded}%
    \renewcommand\transparent[1]{}%
  }%
  \providecommand\rotatebox[2]{#2}%
  \newcommand*\fsize{\dimexpr\f@size pt\relax}%
  \newcommand*\lineheight[1]{\fontsize{\fsize}{#1\fsize}\selectfont}%
  \ifx\svgwidth\undefined%
    \setlength{\unitlength}{86.31763311bp}%
    \ifx\svgscale\undefined%
      \relax%
    \else%
      \setlength{\unitlength}{\unitlength * \real{\svgscale}}%
    \fi%
  \else%
    \setlength{\unitlength}{\svgwidth}%
  \fi%
  \global\let\svgwidth\undefined%
  \global\let\svgscale\undefined%
  \makeatother%
  \begin{picture}(1,0.23187348)%
    \lineheight{1}%
    \setlength\tabcolsep{0pt}%
    \put(0.09522554,0.03615831){\color[rgb]{0,0,0}\makebox(0,0)[lt]{\lineheight{1.25}\smash{\begin{tabular}[t]{l}$\bar{E}_1(-3)$\end{tabular}}}}%
    \put(0,0){\includegraphics[width=\unitlength,page=1]{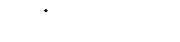}}%
    \put(0.60605719,0.10288551){\color[rgb]{0,0,0}\makebox(0,0)[lt]{\lineheight{1.25}\smash{\begin{tabular}[t]{l}$\bar{E}_2(-2)$\end{tabular}}}}%
    \put(0.34192596,0.12790865){\color[rgb]{0,0,0}\makebox(0,0)[lt]{\lineheight{1.25}\smash{\begin{tabular}[t]{l}$E_3(-1)$\end{tabular}}}}%
    \put(0,0){\includegraphics[width=\unitlength,page=2]{eK4.pdf}}%
  \end{picture}%
\endgroup%
}
  \caption{}
\label{fig:1/4KX}
\end{figure} 

Denote by $q_i$ (resp.\ $E_i$) the blow-up points (resp.\ exceptional curves). Then $l(q_1)=l(q_2)=1$, $l(q_3)=a(q_3)=2$. So each $E_i$ is $\cF'$-invariant and
\[
K_{\cF'}=\sigma^*(K_{\cF})-E_3,\qquad
K_{X'}=\sigma^*(K_X)+\bar{E}_1+2\bar{E}_2+4E_3.
\]
Thus,
\[
K_{\cF'}+\frac{1}{4}K_{X'}
=\sigma^*(K_{\cF}+\frac{1}{4}K_X)
+\frac{1}{4}\bar{E}_1+\frac{1}{2}\bar{E}_2.
\]
By Lemma~\ref{lem:P(D)C=0-C-Finv}, the curves $E_1,E_2,E_3$ are exceptional for the $(K_{\cF'}+\tfrac{1}{4}K_{X'})$-non-positive contraction, corresponding to case~(F). 
In particular, the canonical model $(Y,\cG,\tfrac{1}{4}K_Y)$ has a singularity of $\cG$ which is not log canonical.
\end{example}

Finally, we show that the condition $\lfloor\Delta\rfloor=0$ in Theorem~\ref{thm:Null} is necessary. 
Otherwise, Corollary~\ref{coro:null}(1) may fail: when $\epsilon=0$, $\cG$ need not be canonical (cf.~Remark~\ref{rmk:e=0}).

Consider the special case $D=C$, where $C$ is an irreducible non-$\cF$-invariant curve satisfying
\[
{\rm tang}(\cF,C)=K_{\cF}\cdot C+C^2=0,\qquad C^2<0.
\]
In this situation, a $(C,\cF)$-chain is precisely an $\cF$-chain disjoint from $C$. 

\begin{theorem}\label{thm:NullCtang=0}
Let $(X,\cF,C)$ be a foliated triple, where
\begin{itemize}
\item[\rm (i)] $\cF$ is a foliation on $X$ with at most canonical singularities;
\item[\rm(ii)] $C$ is an irreducible non-$\cF$-invariant curve such that ${\rm tang}(\cF,C)=0$ and $C^2<0$;  
\item[\rm(iii)] $X$ is a smooth projective surface containing no $\cF$-invariant $(-1)$-curves $E$ satisfying 
$K_{\cF}\cdot E\leq0$ and $K_{\cF}\cdot E+2C\cdot E\leq1$.
  \end{itemize}
  Assume that $K_{\cF}+C$ is pseudo-effective, with the Zariski decomposition $P(C)+N(C)$.
Let $Z$ be a connected component of the set ${\rm Exc}(\sigma)$.
Then each irreducible component of $Z$ is $C$ or an $\cF$-invariant curve, and $Z$ falls into one of the following cases:
\begin{itemize}

\item[(1)] $Z=\Gamma_1+\cdots+\Gamma_r$ is an $\cF$-chain disjoint from $C$.

\item[(2)] $Z=\Gamma_1+\cdots+\Gamma_r$ is a chain of smooth rational $\cF$-invariant curves disjoint from $C$, where ${\rm Z}(\cF,\Gamma_i)=2$ for all $i$.

\item[(3)] $Z=\Gamma_1+\Gamma_2+\cdots+\Gamma_r$ is a tree of smooth $\cF$-invariant rational curves disjoint from $C$, where
\begin{itemize}
\item[\rm(i)] $\Gamma_3+\cdots+\Gamma_r$ is a chain with ${\rm Z}(\cF,\Gamma_3)=3$ and ${\rm Z}(\cF,\Gamma_i)=2$ for all $i\ge 4$, in particular $\Gamma_i^2\le -2$ for $i\ge 3$;
\item[\rm(ii)] $\Gamma_1$ and $\Gamma_2$ are maximal $\cF$-chains with $\Gamma_1^2=\Gamma_2^2=-2$, each meeting $\Gamma_3$ transversely at distinct points, and disjoint from $\Gamma_j$ for all $j\ge 4$.
\end{itemize}

\item[(4)] $Z$ is an elliptic Gorenstein leaf disjoint from $C$.

\item[(5)] $Z=C+\sum_{i=1}^s\Theta_i$ ($s\geq0$), where
$\Theta_i=\Gamma_{i1}+\cdots+\Gamma_{i,r(i)}$ is an $\cF$-chain with $C\cdot \Gamma_{i1}=1$ and $C\cdot \Gamma_{ij}=0$ for $j\geq2$, where  $\Gamma_{i1}$ denotes the first curve of $\Theta_i$.     
 (See Figure \ref{fig:CT1T2}.)
\begin{figure}[htpt]
  \centering
  \def\svgwidth{\columnwidth}
  \scalebox{0.4}{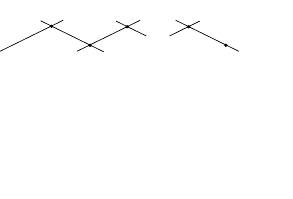}
  \caption{}
\label{fig:CT1T2}
\end{figure} 
\end{itemize}
\end{theorem}
\begin{proof}
Similar to the proof of Theorem \ref{thm:Null}, this result follows from Corollary \ref{coro:KF+Cnoninv} and Lemma \ref{lem:P(D)C=0-C-Finv}.
\end{proof}

\begin{remark}\label{rmk:e=0}
Let $p$ be the point obtained by contracting $Z$. Then cases (1)--(4) correspond to canonical singularities of $\cG$, whereas case (5) corresponds to a log canonical but non-canonical singularity. 
Moreover, case (5) coincides with cases 6 and 7 in \cite[Theorem 1.1]{YAChen23}.
\end{remark}

\subsection{Proof of Theorem \ref{thm:Null}}

\begin{lemma}\label{lem:FnoninvCsuchthatP(Delta)C=0}
Under the assumptions of Theorem~\ref{mainthm:<1/4},
let $C$ be an irreducible non-$\cF$-invariant curve such that the intersection matrix of the components of $C+N(D)$ is negative definite. 
Then the following hold:
\begin{itemize}
\item[\rm(1)] $P(D)\cdot C=0$ if and only if $(K_{\cF}+D)\cdot C=\theta(C)=0$.

\item[\rm(2)] If, in addition, $\lfloor\Delta\rfloor=0$, then $P(D)\cdot C>0$.
\end{itemize}
\end{lemma}
\begin{proof}
By the assumption, we have $C^2\leq\cE(C)^2<0$ (cf. (\ref{eq:defcE(C)})). 

\noindent(1) If $P(D)\cdot C=0$, then we also have $P(D)\cdot \cE(C)=0$.
By Theorem~\ref{mainthm:<1/4},
\[
P(D) = K_{\cF}+D - \sum_{i=1}^s M(D,\Theta_i),
\]
where $\Theta_1,\dots,\Theta_s$ denote all maximal $(D,\cF)$-chains on $X$.
Without loss of generality, assume $\Theta_i\cdot C>0$ for $i\le t$ and $\Theta_i\cdot C=0$ for $i>t$.
By the definitions of $\cE(C)$ and $W_C$ (cf.~\eqref{eq:defcE(C)}), we then have
\[
W_C = K_{\cF}+D - \sum_{i=1}^t M(D,\Theta_i), \qquad
\cE(C)\cdot M(D,\Theta_j) = 0 \quad \text{for all } j.
\]
Hence, 
\[
(P(D) + \theta(C)\,\cE(C))\cdot \cE(C) = (W_C + \theta(C)\,\cE(C))\cdot \cE(C) \ge 0,
\]
where the last inequality follows from Lemmas~\ref{lem:(W+cC)cC>=0} and~\ref{lem:Dtang-index}.
Since $\cE(C)^2 < 0$, we must have $\theta(C)=0$, 
which implies $C\cdot \Theta_i = 0$ for all $i$ and therefore $C\cdot N(D) = 0$.
Consequently,
$(K_{\cF}+D)\cdot C = P(D)\cdot C = 0$.
The converse is immediate.

\noindent(2) If $\lfloor\Delta\rfloor=0$, then $(K_{\cF}+D)\cdot C>0$, by Corollary \ref{cor:(KF+D)C>0-C-nonFinv}.
Therefore, $P(D)\cdot C>0$ by (1).
\end{proof}

\begin{corollary}\label{coro:KF+Cnoninv}
Under the assumption in Theorem \ref{thm:NullCtang=0}, we have $P(C)\cdot C=0$.
\end{corollary}

\begin{lemma}\label{lem:P(D)C=0-C-Finv}
Under the assumptions of Theorem~\ref{mainthm:<1/4}, 
assume moreover that $\cF$ has at most canonical singularities when $\epsilon=0$. 
Let $C$ be an $\cF$-invariant irreducible curve not contained in any maximal $(D,\cF)$-chain. 
If $P(D)\cdot C=0$ and the intersection matrix of $C+N(D)$ is negative definite, 
then one of the following cases occurs:
\begin{itemize}
\item[(A)] $C$ is a rational nodal curve with $D\cdot C=0$ such that 
${\rm Sing}(\cF)\cap C$ consists only of the node of $C$, a non-degenerate reduced singularity of $\cF$.

\item[(B)] $C$ is a smooth rational curve with $K_{\cF}\cdot C=-1$ and $D\cdot C =1$.

\item[(C)] $C$ is a smooth rational curve with $K_{\cF}\cdot C=D\cdot C =0$, disjoint from any maximal $(D,\cF)$-chain.

\item[(D)] $C$ is a smooth rational curve satisfying that there is  a maximal $(D,\cF)$-chain $\Theta$ such that $\Theta+C$ is an $\cF$-chain with $M(D,\Theta)\cdot C=D\cdot C\,(\leq \frac{1}{2})$. 
    
\item[(E)] $C$ is a smooth rational curve with $K_{\cF}\cdot C=1$ and $D\cdot C=0$, satisfying the following: 
    there exist exactly two connected components of ${\rm Supp}(N(D))$, 
    each of which consists of a smooth rational curve $\Gamma_i$ with $\Gamma_i^2=-2$ and $D\cdot \Gamma_i=0$ for $i=1,2$. 
    Moreover, both $\Gamma_1$ and $\Gamma_2$ meet $C$ transversely at the points $p_1$ and $p_2$, and 
    $N(D)|_C=\frac{1}{2}p_1+\frac{1}{2}p_2$. 
    
\item[(F)]  $C$ is a smooth rational curve with $K_{\cF}\cdot C=1$ and $D\cdot C=-\frac{1}{4}$, satisfying the following: 
    there exist exactly two connected components $\Theta_1=\Gamma_1$, $\Theta_2=\Gamma_2$ of ${\rm Supp}(N(D))$,  
    where $\Gamma_1^2=-2,\Gamma_2^2=-3$ and $D\cdot \Gamma_1=0$, $D\cdot \Gamma_2=\frac{1}{4}$. 
 Moreover, both $\Gamma_1$ and $\Gamma_2$ meet $C$ transversely at the points $p_1$ and $p_2$, 
 and $N(D)|_C=\frac{1}{2}p_1+\frac{1}{4}p_2$.  
  \end{itemize} 
In particular, in cases {\rm(A)}--{\rm(E)}, either $C^2 \le -2$ or $p_a(C)=1$. 
In case {\rm(F)}, we have $C^2=-1$, and this occurs only when $\epsilon=1/4$.
\end{lemma}
\begin{proof} 
Let $k$ be the number of points at which $C$ intersects the support of $N(D)$.
Using the notation from the proof of Theorem~\ref{mainthm:<1/4}, we have
\[
h(\cF,C)=\sum_{i=1}^h h_{q_i}(\cF,C)=K_{\cF}\cdot C + 2 - 2g(C) - k \ge 1.
\] 
Since $(K_{\cF}+D)\cdot C = N(D)\cdot C$, it follows that
\begin{equation}\label{equ:N(D)C<=k/2}
\frac{k}{2} \ge N(D)\cdot C = 2g(C)-2 + h(\cF,C) + k + D\cdot C \ge 2g(C)-1 + k + D\cdot C.
\end{equation}

By the definition of $D=D_{\Delta,\epsilon}$ (cf.~Definition~\ref{def:D(Delta,epsilon)}), we have 
\begin{equation}\label{ineq:DC}
D\cdot C \ge \epsilon K_X\cdot C \ge -\epsilon \ge -\tfrac{1}{4},
\end{equation}
and $D\cdot C<0$ can occur only if $C$ is a $(-1)$-curve.  
Hence \eqref{equ:N(D)C<=k/2} and \eqref{ineq:DC} imply that $g(C)=0$ and $h(\cF,C)\in\{1,2\}$.
\medskip

Suppose that $p_a(C)\ge 1$. Then $D\cdot C \ge 0$.  
If $h(\cF,C)=1$, then by Lemma~\ref{lem:CSqi<0}(2), the curve $C$ is smooth, hence $p_a(C)=g(C)=0$, a contradiction.  
Thus we may assume that $h(\cF,C)=2$. 
Then \eqref{equ:N(D)C<=k/2} implies that $k=0$ and $D\cdot C=0$.

If $\epsilon>0$, then $K_X\cdot C\leq0$ by \eqref{ineq:DC}, and hence
$p_a(C)\leq\tfrac{1}{2}(C^2+2)<1$,
a contradiction.  
Thus $\epsilon=0$, and by assumption $\cF$ has at most canonical singularities.

If $h=2$, then $h_{q_1}(\cF,C)=h_{q_2}(\cF,C)=1$. By Proposition~\ref{prop:hp(F,C)}(3), the curve $C$ is smooth at each $q_i$, hence smooth, so $p_a(C)=g(C)=0$, a contradiction.  

If $h=1$, then $h_{q_1}(\cF,C)=2$ and $m_{q_1}(C)\ge 2$. By Proposition~\ref{prop:hp(F,C)}(4), the point $q_1$ is a node of $C$ and a non-degenerate reduced singularity of $\cF$.  
Hence $p_a(C)=1$, corresponding to case~(A).
\medskip

Next, assume $p_a(C)=0$. Then 
$N(D)\cdot C = -2 + h(\cF,C) + k + D\cdot C$.
Since $0 \le N(D)\cdot C \le k/2$, we have
\begin{equation}\label{ineqs:N(D)C}
\frac{k}{2} \ge -2 + h(\cF,C) + k + D\cdot C \ge 0.
\end{equation}
Since $h(\cF,C)\ge 1$, $k\ge 0$ and $D\cdot C\ge -1/4$, \eqref{ineqs:N(D)C} gives the following possibilities:
\[
(k, h(\cF,C)) = (0,1), \ (0,2), \ (1,1), \ (2,1).
\]

\noindent{\bf Case $(0,1)$.} Then $k=0$, $K_{\cF}\cdot C=-1$, and $D\cdot C=1$, corresponding to case~(B).
In this case, $C^2\leq-2$ due to condition (2a) of Theorem~\ref{mainthm:<1/4}.

\smallskip

\noindent{\bf Case $(0,2)$.} Then $k=0$, $K_{\cF}\cdot C=D\cdot C=0$, and $C$ is disjoint from $N(D)$, corresponding to case~(C).
In this case, $C^2\leq-2$ due to condition (2a) of Theorem~\ref{mainthm:<1/4}.

\smallskip

\noindent{\bf Case $(1,1)$.} Then $k=1$, $K_{\cF}\cdot C=0$, and $0\leq D\cdot C\le 1/2$.  
By condition (2a) of Theorem~\ref{mainthm:<1/4}, $C^2 \le -2$.  
Moreover,  $\Theta_1+C$ is an $\cF$-chain and $M(D,\Theta_1)\cdot C=N(D)\cdot C=D\cdot C$,
corresponding to case~(D).

\smallskip

\noindent{\bf Case $(2,1)$.} Then $k=2$, $K_{\cF}\cdot C=1$, and $-1/4 \le D\cdot C \le 0$.  
Thus,
\begin{equation}\label{equ:1+DC<=}
\frac{3}{4}\leq 1 + D\cdot C = N(D)\cdot C = \sum_{i=1}^2 M(D,\Theta_i)\cdot C \le \frac{1}{n_1} + \frac{1}{n_2},
\end{equation}
where $\Theta_1$ and $\Theta_2$ are maximal $(D,\cF)$-chains meeting $C$, using the notation from the proof of Theorem~\ref{mainthm:<1/4}.

If $D\cdot C=0$, then $N(D)\cdot C=1$ and $M(D,\Theta_1)\cdot C=M(D,\Theta_2)\cdot C=1/2$.  
It follows that $\Theta_i=\Gamma_i$, where $\Gamma_i^2=-2$ and $D\cdot \Gamma_i=0$ for $i=1,2$, corresponding to case~(E).
In this case, we must have $C^2\le -2$; otherwise, the intersection matrix of $C+\Gamma_1+\Gamma_2$ is not negative definite, a contradiction.

If $-\tfrac{1}{4} \le D\cdot C < 0$, then $\epsilon>0$ and $C$ is a $(-1)$-curve.  
By \eqref{equ:1+DC<=}, we have $(n_1,n_2)=(2,2), (2,3)$, or $(2,4)$.  
As in the last part of the proof of Theorem~\ref{mainthm:<1/4}, the case $(n_1,n_2)=(2,2)$ cannot occur, 
while the case $(n_1,n_2)=(2,3)$ occurs only if $D\cdot C = -\epsilon = -\tfrac{1}{4}$, which corresponds to case~(F). 

If $(n_1,n_2)=(2,4)$, then $\Theta_1=\Gamma_{11}$ with $\Gamma_{11}^2=-2$, and $\Theta_2=\sum_{j=1}^{r(2)}\Gamma_{2j}$.  
Since the intersection matrix of $C+\Gamma_{11}+\Gamma_{2,r(2)}$ is negative definite, we have $\Gamma_{2,r(2)}^2\le -3$.  
Combining this with $n_2=4$, we obtain $r(2)=1$ and $\Gamma_{21}^2=-4$.
By \eqref{equ:1+DC<=}, we obtain
\[
\frac{1}{4}
= M(D,\Theta_2)\cdot C
= \gamma_{21}
= \frac{(K_{\cF}+D)\cdot \Gamma_{21}}{-\Gamma_{21}^2}
= \frac{1 - D\cdot \Gamma_{21}}{4}.
\]
Thus $D\cdot \Gamma_{21}=0$, hence $K_X\cdot \Gamma_{21}=0$, which forces $\Gamma_{21}^2=-2$, a contradiction.
\end{proof}

\begin{proof}[Proof of Theorem \ref{thm:Null}]
By Theorem~\ref{mainthm:<1/4} and Lemmas~\ref{lem:FnoninvCsuchthatP(Delta)C=0} 
and~\ref{lem:P(D)C=0-C-Finv}, $Z$ consists of maximal $(D,\cF)$-chains and $\cF$-invariant curves of type {\rm(A)}--{\rm(E)}. 
The desired cases follow, while all other possibilities are excluded both by the negative definiteness of the intersection matrix of $Z$ and by the separatrix theorem (cf.~Theorem~\ref{thm:separatrix}).
\end{proof}

\section{Effective behavior of the linear system \texorpdfstring{$|m(K_{\cF}+D)|$}{|m(KF+D)|}}
Now assume that $K_{\cF}+D$ is big. 
We consider the contraction morphism associated with ${\rm Null}\,P(D)$:
\begin{equation}
\sigma:(X,\cF,D)\to(Y,\cG,D_Y),
\end{equation}
where $\sigma_*\cO_X=\cO_Y$, $K_{\cG}=\sigma_*K_{\cF}$ and $D_Y=\sigma_*D$.
Then $K_{\cG}+D_Y$ is \emph{numerically ample}, i.e., $(K_{\cG}+D_Y)^2>0$ and $(K_{\cG}+D_Y)\cdot C>0$ for every irreducible curve $C$ on $Y$.

In this section, we study the boundedness of the adjoint divisor $K_{\cG}+D_Y$, 
using Tan's results on the effective behavior of multiple linear systems on surfaces \cite{Tan04}.

\subsection{Tan's result and its corollaries}
For two divisors $A$ and $T$, set
\begin{equation}
\mathfrak{M}(A,T):=\frac{\left((K_X-T)\cdot A+2\right)^2}{4A^2}-\frac{(K_X-T)^2}{4}.
\end{equation}
\begin{theorem}[Tan]\label{them:Tan}
Let $A$ be a nef and big divisor on a smooth projective surface $X$, let $T$ be any fixed divisor, and let $k$ be a nonnegative integer. 
Assume that either $n>k +\mathfrak{M}(A, T)$, or $n\geq\mathfrak{M}(A, T)$ when $k=0$ and $T\sim K_X +\lambda A$ for some $\lambda\in\bQ$.
Suppose $|nA + T|$ is not $(k-1)$-very ample, i.e., there is a zero dimensional subscheme $\Delta$ on $X$ with minimal degree
$\deg \Delta\leq k$ such that it does not give independent conditions on $|nA + T|$. 
Then there is an effective divisor $D\neq0$ containing $\Delta$ such that
\begin{equation}
\begin{cases}
D^2+K_X\cdot D-T\cdot D\geq -k,\\
D\cdot A=0.
\end{cases}
\end{equation}
(See \cite[Theorem~1.2]{Tan04}.)
\end{theorem}

In Theorem~\ref{them:Tan}, we use the convention that a divisor $D$ (or the linear system $|D|$) is $(-1)$-very ample if $H^1(\cO_X(D))=0$, 
$0$-very ample if it is base-point-free, and $1$-very ample if it is very ample.

Denote by ${\rm Null}(A)=\{C\mid C\cdot A=0\}$ the exceptional locus of $A$.
Its connected components are contracted to normal surface singularities $Q_1,\dots,Q_s$. 
Let $E_{Q_i}$ be the exceptional locus over $Q_i$; then ${\rm Null}(A)=\bigcup_i E_{Q_i}$. 
Let $Z_{Q_i}$ be the \emph{(numerical) fundamental cycle} of $E_{Q_i}$, 
i.e., the minimal effective divisor supported on $E_{Q_i}$ such that 
$Z_{Q_i}\cdot C \le 0$ for every irreducible component $C$ of $E_{Q_i}$ (cf.~\cite[p.~84]{Reid}).

\medskip

Assume that each $Q_i$ is either a rational singularity or an elliptic Gorenstein singularity.
Set 
$p_a(Q):=p_a(Z_Q)$ and 
\begin{equation}
  \cZ_i:=\sum_{p_a(Q)=i}Z_Q,\quad i=0,1.
\end{equation}

The following lemma can be found in \cite[Sections~4.10, 4.12, and 4.21]{Reid}.
\begin{lemma}\label{lem:Reid}
Let $Q$ be a normal surface singularity.
\begin{itemize}
  \item[\rm(i)] If $Q$ is a rational singularity, then $p_a(Z_Q)=0$ and $p_a(D)\leq 0$ for every nonzero effective divisor $D$ supported on $E_Q$.
  
  \item[\rm(ii)] If $Q$ is an elliptic Gorenstein singularity, then $p_a(Z_Q)=1$ and $p_a(D)\leq 0$ for every nonzero effective divisor
$D\neq Z_Q$ supported on $E_Q$, and $\omega_{Z_Q}=\cO_{Z_Q}$.
\end{itemize}
\end{lemma}

\begin{proposition}\label{prop:nA-2Z1bpf}
Let $A$ be a nef and big divisor on a smooth projective surface $X$. 
Write ${\rm Null}(A)=\bigcup_{i=1}^s E_{Q_i}$, where each $Q_i$ is either a rational singularity or an elliptic Gorenstein singularity.
\begin{itemize}
\item[(1)] If $n>\mathfrak{M}(A,0)$, then $|nA-2\cZ_1|$ is $(-1)$-very ample, i.e.,
\[
H^1(X,\cO(nA-2\cZ_1))=0.
\]

\item[(2)] If $n>1+\mathfrak{M}(A,0)$, then $|nA-2\cZ_1|$ is  $0$-very ample, i.e., base-point-free. 
In particular, it defines a morphism
\[
\Phi_n=\Phi_{|nA-2\cZ_1|}:X\longrightarrow \bP^N.
\]
\end{itemize}
\end{proposition}
\begin{proof}
With the notation of Theorem~\ref{them:Tan}, write $D=\sum_i D_i>0$, where each $D_i$ is an effective divisor supported on $E_{Q_i}$. 
By Lemma~\ref{lem:Reid}, for each $D_i\neq0$, we have
\[
\begin{cases}
D_i^2+K_X\cdot D_i=0,\quad \cZ_1\cdot D_i=Z_{Q_i}^2\leq -1, 
& \text{if } D_i=Z_{Q_i} \text{ and } p_a(Q_i)=1,\\
D_i^2+K_X\cdot D_i\leq -2,\quad \cZ_1\cdot D_i\leq 0, 
& \text{otherwise}.
\end{cases}
\]
In particular, in all cases we obtain
\begin{equation}\label{ineq:D2-KXD+2ZD}
D_i^2+K_X\cdot D_i+2\cZ_1\cdot D_i\leq -2.
\end{equation}
Since $D=\sum_i D_i\neq 0$ and $D_i\cdot D_j=0$ for $i\neq j$, summing \eqref{ineq:D2-KXD+2ZD} gives
\[
D^2+K_X\cdot D+2\cZ_1\cdot D
= \sum_i \big(D_i^2+K_X\cdot D_i+2\cZ_1\cdot D_i\big)
\leq -2.
\]
Therefore, by Theorem~\ref{them:Tan}, the linear system $|nA-2\cZ_1|$ is $(-1)$-very ample if
\[
n>\mathfrak{M}(A,-2\cZ_1)=\mathfrak{M}(A,0),
\]
and it is $0$-very ample if $n>1+\mathfrak{M}(A,0)$.
\end{proof}

The following proposition summarizes properties of $\Phi_n$, with statements (1)--(4) analogous to \cite[Theorem~4.8]{Tan04}.
\begin{proposition}\label{prop:nA-2Z1}
Under the assumption of Proposition \ref{prop:nA-2Z1bpf}, if  
\[
n>2+\mathfrak{M}(A,0),
\]
then 
\begin{itemize}
  \item[(1)] $\Phi_n$ is a birational morphism onto a projective surface $\Sigma_n=\Phi_{n}(X)$.
  \item[(2)] On $X\setminus {\rm Null}(A)$, $\Phi_n$ is an isomorphism.
  \item[(3)] $\Phi_n$ contracts each curve $E_Q$ with $p_a(Q)=0$ to a (singular) point of $\Sigma_n$.
  \item[(4)] 
  $\Phi_n(E_{Q_{\alpha}})\neq\Phi_n(E_{Q_{\beta}})$ if $p_a(Q_{\alpha})=p_a(Q_{\beta})=0$ and $\alpha\neq\beta$.
  \item[(5)] On $X\setminus \cZ_1$, $\Phi_n$ is just the contraction morphism of the set $\bigcup_{p_a(Q)=0}E_Q$.
\end{itemize}
\end{proposition}

\begin{proof}
By Theorem~\ref{them:Tan} and Proposition~\ref{prop:nA-2Z1bpf}, the linear system $|nA-2\cZ_1|$ is base-point-free, and the locus where it fails to be very ample is contained in ${\rm Null}(A)$. 
Hence (1) and (2) follow.  
Since $(nA-2\cZ_1)\cdot Z_{Q_i}=0$ for $p_a(Q_i)=0$, (3) also holds.

\medskip

\noindent(4) Let $\cE=Z_{Q_{\alpha}}+Z_{Q_{\beta}}$, where $p_a(Q_{\alpha})=p_a(Q_{\beta})=0$ and $\alpha\neq\beta$. 
For any $D_i\neq 0$, we have $\cE\cdot D_i\le 0$, so by \eqref{ineq:D2-KXD+2ZD},
\[
D_i^2+K_X\cdot D_i+(2\cZ_1+\cE)\cdot D_i\le -2.
\]
Summing over $i$, we obtain
\[
D^2+K_X\cdot D+(2\cZ_1+\cE)\cdot D\le -2,
\]
where $D=\sum_i D_i>0$ as above. 
By Theorem~\ref{them:Tan}, $|nA-2\cZ_1-\cE|$ is $(-1)$-very ample if
\[
n>\mathfrak{M}(A,-2\cZ_1-\cE)
=2+\mathfrak{M}(A,0)-\frac{-Z_{Q_{\alpha}}^2-Z_{Q_{\beta}}^2}{4}.
\]
Under the assumption $n>2+\mathfrak{M}(A,0)$, this holds, hence $H^1(nA-2\cZ_1-\cE)=0$.

Consider the exact sequence
\begin{equation}\label{seq:nA-E-2Z}
0\to\cO(nA-2\cZ_1-\cE)\to\cO(nA-2\cZ_1)\to \cO_{\cE}(nA-2\cZ_1)\to0.
\end{equation}
Since $|nA-2\cZ_1|$ is base-point-free and $(nA-2\cZ_1)\cdot E=0$ for every component $E$ of $\bigcup_{p_a(Q)=0}E_Q$, 
a general member $G\in|nA-2\cZ_1|$ does not contain any such $E$, hence is disjoint from $\cE$. 
Thus
\[
\cO_{\cE}(nA-2\cZ_1)=\cO_{\cE}(G)=\cO_{\cE}
=\cO_{Z_{Q_{\alpha}}}\oplus\cO_{Z_{Q_{\beta}}}.
\]
Since $H^1(nA-2\cZ_1-\cE)=0$, \eqref{seq:nA-E-2Z} yields a surjection
\[
H^0(nA-2\cZ_1)\twoheadrightarrow 
H^0(\cO_{Z_{Q_{\alpha}}})\oplus H^0(\cO_{Z_{Q_{\beta}}}).
\]
Hence there exist $s_1,s_2\in H^0(nA-2\cZ_1)$ such that
\[
s_1|_{Z_{Q_\alpha}}\neq 0,\ s_1|_{Z_{Q_\beta}}=0,\quad
s_2|_{Z_{Q_\alpha}}=0,\ s_2|_{Z_{Q_\beta}}\neq 0.
\]
Thus $|nA-2\cZ_1|$ separates $E_{Q_{\alpha}}$ and $E_{Q_{\beta}}$, and hence $\Phi(E_{Q_{\alpha}})\neq\Phi(E_{Q_{\beta}})$.

\medskip

\noindent(5) It suffices to show that $\big((\Phi_n)_*\cO_X\big)_q=(\cO_{\Sigma_n})_q$ for any $q\in \Sigma_n\setminus\Phi_n(\cZ_1)$. 
By (2), it suffices to consider the case that $q=\Phi_n(E_Q)$ with $p_a(Q)=0$. 
Let $B=(\cO_{\Sigma_n})_q$ with maximal ideal $\mathfrak{m}$, and set $B'=\big((\Phi_n)_*\cO_X\big)_q$. 
Then $B\subset B'$ is a ring extension. 
By (4), there is a unique maximal ideal $\mathfrak{m}'$ of $B'$ containing $\mathfrak{m}$, so $(B',\mathfrak{m}')$ is local. 
It remains to show $B=B'$.

Since $H^1(\cO_{Z_Q})=0$ and $(nA-2\cZ_1)\cdot Z_Q=0$, it follows from \cite[Proposition~4.14]{Reid} that $\cO(nA-2\cZ_1)|_{Z_Q}=\cO_{Z_Q}$. 
Thus we have an exact sequence
\[
0\to\cO(nA-2\cZ_1)\otimes\cI_{Z_Q}^2\to \cO(nA-2\cZ_1)\otimes\cI_{Z_Q}\to\cI_{Z_Q}/\cI_{Z_Q}^2\to0.
\]
As in (4), $|nA-2\cZ_1-2Z_Q|$ is $(-1)$-very ample, hence
\[
H^1(\cO(nA-2\cZ_1)\otimes\cI_{Z_Q}^2)=H^1(nA-2\cZ_1-2Z_Q)=0.
\]
This yields a surjection
\[
H^0(\cO_X(nA-2\cZ_1)\otimes\cI_{Z_Q})\twoheadrightarrow H^0(\cI_{Z_Q}/\cI_{Z_Q}^2).
\]
On the other hand, since $Q$ is rational, we have
\[
H^0(\cI_{Z_Q}/\cI_{Z_Q}^2)\cong\mathfrak{m}'/(\mathfrak{m}')^2
\]
by \cite[III, Proposition~(3.8)]{BPV04} or \cite[Theorem~4.17]{Reid}. 
Since the map $H^0(\cO_X(nA-2\cZ_1)\otimes\cI_{Z_Q})\to \mathfrak{m}'/(\mathfrak{m}')^2$  factors through $\mathfrak{m}$, we obtain a surjection
\[
\mathfrak{m}\twoheadrightarrow\mathfrak{m}'/(\mathfrak{m}')^2.
\]
Thus $\mathfrak{m}'=\mathfrak{m}B'$ by \cite[Proposition~2.8]{AM69}, and hence $B'=B+\mathfrak{m}B'$, which implies $B'=B$ by \cite[Corollary~2.7]{AM69}.
\end{proof}

\begin{lemma}\label{lem:H1(O(mA))=0}
Under the assumptions of Proposition~\ref{prop:nA-2Z1bpf}, let $Z=Z_Q$ with $p_a(Q)=1$. 
If, in addition, $A|_Z$ is not torsion on $Z$, then for any $m>0$,
\[
H^0\bigl(Z,\cO_Z(mA)\bigr)=
H^1\bigl(\cO_{2Z}(mA)\bigr)=0.
\]
\end{lemma}

\begin{proof}
Since $\omega_Z \cong \cO_Z$ (Lemma~\ref{lem:Reid}), Serre duality gives
\[
h^1(\cO_Z(mA)) = h^0(\cO_Z(-mA)) = 0,
\]
as $A\cdot Z = 0$ and $A|_Z$ is non-torsion.

From the exact sequence (cf.~\cite[Ch.~II, Sec.~1, (4)]{BPV04})
\[
0 \to \cO_Z(mA-Z) \to \cO_{2Z}(mA) \to \cO_Z(mA) \to 0,
\]
the associated cohomology sequence gives
\[
H^1(\cO_Z(mA-Z)) \to H^1(\cO_{2Z}(mA)) \to H^1(\cO_Z(mA))=0.
\]
Since  $\deg_Z(Z-mA)=Z^2<0$ and  $\omega_Z \cong \cO_Z$, Serre duality yields 
\[
h^1(\cO_Z(mA-Z)) = h^0(\cO_Z(Z-mA)) = 0.
\] 
Thus $H^1(\cO_{2Z}(mA)) = 0$.
\end{proof}

\begin{proposition}\label{prop:H1vanishing}
Under the assumptions of Proposition~\ref{prop:nA-2Z1bpf}, assume in addition that 
$A|_Z$ is not torsion on $Z$ for any $Z=Z_Q$ with $p_a(Q)=1$. 
Then
\[
H^i\bigl(X,\cO_X(mA)\bigr)=0,\qquad i=1,2,
\]
for all $m>\mathfrak{M}(A,0)$.
\end{proposition}
\begin{proof}
Let $m>\mathfrak{M}(A,0)$. 
By Lemma~\ref{lem:H1(O(mA))=0} and Proposition~\ref{prop:nA-2Z1bpf}, we have
\begin{equation}\label{equ:H1=0}
H^1\bigl(X,\cO_X(mA-2\cZ_1)\bigr)
=
H^1\bigl(\cO_{2\cZ_1}(mA)\bigr)
=0.
\end{equation}
Consider the exact sequence
\[
0 \to \cO_X(mA-2\cZ_1)
\to \cO_X(mA)
\to \cO_{2\cZ_1}(mA)
\to 0.
\]
The associated long exact sequence in cohomology yields
\[
H^1\bigl(X,\cO_X(mA-2\cZ_1)\bigr)
\to
H^1\bigl(X,\cO_X(mA)\bigr)
\to
H^1\bigl(\cO_{2\cZ_1}(mA)\bigr),
\]
and hence $H^1\bigl(X,\cO_X(mA)\bigr)=0$ by \eqref{equ:H1=0}.

For $H^2$, by the Hodge index theorem (cf.~Lemma~\ref{lem:Hodgeindex}) we have
\[
\mathfrak{M}(A,0)-\frac{K_X\cdot A}{A^2}
=
\frac{(K_X\cdot A)^2-K_X^2A^2+4}{4A^2}
>0.
\]
Thus, for $m>\mathfrak{M}(A,0)$, we obtain $(K_X-mA)\cdot A<0$, and hence
\[
H^2\bigl(X,\cO_X(mA)\bigr)
=
H^0\bigl(X,\cO_X(K_X-mA)\bigr)
=0.
\]
\end{proof}

\subsection{Effective behavior of \texorpdfstring{$|m (K_{\cG}+D_Y)|$}{KG+DY}}

Set 
\begin{equation}
\alpha(D,\cF):=\left\lfloor\frac{\left(K_X\cdot A+2\right)^2}{4A^2}+\frac{9A^2+6K_X\cdot A}{4}\right\rfloor,
\end{equation}
where $A:=i(D,\cF) P(D)$ and $i(D,\cF)$ is the smallest positive integer $m$ such that $m P(D)$ is integral.

\begin{theorem}\label{them:boundednessofD}
Let $(X,\cF,D=D_{\Delta,\epsilon})$ be as in the setting of Theorem~\ref{thm:Null}, 
and assume that $K_{\cF}+D$ is big. 
Let $(Y,\cG,D_Y)$ be the canonical model of $(X,\cF,D)$.

If $m\ge i(D,\cF)\bigl(\alpha(D,\cF)+3\bigr)$ and $m$ is divisible by $i(D,\cF)$, then 
the linear system $|m(K_{\cG}+D_Y)|$ defines a rational map
\[
\phi_m \colon Y \dashrightarrow \mathbb{P}^N
= \mathbb{P} H^0\bigl(Y, m(K_{\cG}+D_Y)\bigr),
\]
which is an isomorphism onto its image outside the elliptic Gorenstein singularities.
\end{theorem}
\begin{proof}
  Similar to the claim in the proof of Theorem \ref{thm:eKXcanonical}, we can conclude that $3A+K_X$ is nef. 
  Thus, we have
  $$0\leq(3A+K_X)^2=K_X^2+9A^2+6A\cdot K_X,$$
  which implies
  \begin{equation}
  \frac{\left(K_X\cdot A+2\right)^2}{4A^2}+\frac{9A^2+6K_X\cdot A}{4}\geq \frac{\left(K_X\cdot A+2\right)^2}{4A^2}-\frac{K_X^2}{4}=\mathfrak{M}(A,0).\notag
  \end{equation}
  So
  \begin{equation}\label{ineq:M(A,0)}
  \alpha(D,\cF)=\left\lfloor\frac{\left(K_X\cdot A+2\right)^2}{4A^2}+\frac{9A^2+6K_X\cdot A}{4}\right\rfloor> \mathfrak{M}(A,0)-1. 
  \end{equation}

The statement now follows from Theorem~\ref{thm:Null} and Proposition~\ref{prop:nA-2Z1}.
\end{proof}

\begin{corollary}\label{coro:boundednessofeKX}
Let $(X,\cF)$ be a log minimal foliated surface (cf.~Definition~\ref{def:logminfoliation}).
Assume $\epsilon\in(0,\frac{1}{4})\cap\bQ$ and that $K_{\cF}+\epsilon K_X$ is big with Zariski decomposition $P(\epsilon)+N(\epsilon)$.
Let $(Y,\cG,\epsilon K_Y)$ be the canonical model of $(X,\cF,\epsilon K_X)$. 

If, in addition, $K_{\cF}$ is pseudo-effective, then there exists a positive integer $\mathfrak{n}$, depending only on $\epsilon$, $i(\cF)$ and ${\rm Vol}(K_{\cF}+\epsilon K_X)$, such that
$\mathfrak{n} (K_{\cG}+\epsilon K_Y)$ is very ample.
\end{corollary}
\begin{proof}
Let \(i=i(\epsilon K_X,\cF)\) and \(A=iP(\epsilon)\).
Similar to the claim in the proof of Theorem~\ref{thm:eKXcanonical}, we have \(K_X+3A\) is nef. 
In particular, \((K_X+3A)\cdot A \ge 0\), hence \(K_X\cdot A \ge -3A^2\).
On the other hand, since \(K_{\cF}\) is pseudo-effective and \(A\) is nef, we have \(A\cdot K_{\cF}\ge 0\). Thus
\[
K_X\cdot A
=
\frac{i}{\epsilon}P(\epsilon)\cdot\big(P(\epsilon)+N(\epsilon)-K_{\cF}\big)
=
\frac{i}{\epsilon}\big(P(\epsilon)^2 - P(\epsilon)\cdot K_{\cF}\big)
\le \frac{i}{\epsilon}P(\epsilon)^2
=
\frac{A^2}{\epsilon i}.
\]
Therefore,
\[
-3A^2 \le K_X\cdot A \le \frac{A^2}{\epsilon i}.
\]
It follows that
\[
\mathfrak{n}:=i\cdot\left\lfloor
\frac{(3\epsilon i+1)^2}{4\epsilon^2 i^2}A^2
+\frac{1}{A^2}
+\frac{1}{\epsilon i}
\right\rfloor+3i
\ge i\cdot (\alpha(\epsilon K_X,\cF)+3).
\]

Finally, note that \(i(\epsilon K_X,\cF)\) depends only on \(\epsilon\) and \(i(\cF)\), and
\[
A^2=i^2P(\epsilon)^2=i^2\,{\rm Vol}(K_{\cF}+\epsilon K_X).
\]
Thus \(\mathfrak n\) depends only on \(\epsilon\), \(i(\cF)\) and \({\rm Vol}(K_{\cF}+\epsilon K_X)\). 

The conclusion now follows from Theorem~\ref{them:boundednessofD}.
\end{proof}

Finally, we establish the following vanishing result based on Proposition~\ref{prop:H1vanishing}.

\begin{proposition}\label{prop:P(D)H1}
Let $\cF$ be a canonical foliation on a smooth surface $X$, and let 
$\Delta=\sum_{i=1}^l a_i C_i$, where $a_i\in\bQ\cap[0,1)$ and each $C_i$ is not $\cF$-invariant.
Assume that there exists no $\cF$-invariant $(-1)$-curve $\Gamma$ such that
$K_{\cF}\cdot\Gamma\le 0$ and $K_{\cF}\cdot\Gamma+2\Delta\cdot\Gamma\le 1$.

If $K_{\cF}+\Delta$ is big with Zariski decomposition $K_{\cF}+\Delta=P(\Delta)+N(\Delta)$, 
then for every integer 
$
m\ge i(\Delta,\cF)\bigl(\alpha(\Delta,\cF)+1\bigr)
$
divisible by $i(\Delta,\cF)$, we have
\[
H^i\bigl(X,mP(\Delta)\bigr)=0,\qquad i=1,2.
\]
Moreover,
\[
\dim H^0\bigl(X,m(K_{\cF}+\Delta)\bigr)
=
\frac{m^2}{2}\operatorname{Vol}(K_{\cF}+\Delta)
-\frac{m}{2}K_X\cdot P(\Delta)
+\chi(\cO_X).
\]
\end{proposition}

\begin{proof}
Let $E_Q$ be a connected component of ${\rm Null}(P(\Delta))$ with $p_a(Q)=1$. 
By Theorem~\ref{thm:Null}, $E_Q$ is an elliptic Gorenstein leaf, and 
$\Delta\cdot C=0$ for every irreducible component $C$ of $E_Q$. 
In particular, if $\Delta\neq 0$, then $\Delta$ is disjoint from $E_Q$. Hence
\[
P(\Delta)|_{Z_Q}=(K_{\cF}+\Delta)|_{Z_Q}=K_{\cF}|_{Z_Q},
\]
where $Z_Q$ denotes the fundamental cycle of $E_Q$. 
By \cite[Fact~III.0.4 and Theorem~IV.2.2]{MMcQ08}, this divisor is not torsion on $Z_Q$.

The vanishing $H^i(X,mP(\Delta))=0$ for $i=1,2$ follows from Proposition~\ref{prop:H1vanishing} and \eqref{ineq:M(A,0)}. 
Hence the stated formula for $h^0$ follows from Riemann--Roch, noting that $h^0(m(K_{\cF}+\Delta))=h^0(mP(\Delta))$.
\end{proof}

Similarly, we obtain the following effective version of \cite[Theorem~1.6]{JhL}.

\begin{proposition}\label{Prop:Jihao}
Let $\cF$ be a canonical foliation on a smooth surface $X$, and let 
$\Delta=\sum_{i=1}^l a_i C_i$, where $a_i\in[0,1)\cap\bQ$ and each $C_i$ is not $\cF$-invariant.
Assume that $K_{\cF}+\Delta$ is nef and big. 

Then, for any integer $m\ge i(\Delta,\cF)\bigl(\lfloor\mathfrak{M}(A,0)\rfloor+1\bigr)$ 
such that $m$ is divisible by $i(\Delta,\cF)$, where 
$A=i(\Delta,\cF)(K_{\cF}+\Delta)$, we have
\[
H^i\bigl(X, m(K_{\cF}+\Delta)\bigr)=0,\qquad i=1,2.
\]
Moreover,
\[
\dim H^0\bigl(X, m(K_{\cF}+\Delta)\bigr)
=
\frac{m^2}{2}(K_{\cF}+\Delta)^2
-\frac{m}{2}K_X\cdot (K_{\cF}+\Delta)
+\chi(\cO_X).
\]
\end{proposition}
\begin{proof}
This is a direct consequence of Proposition~\ref{prop:H1vanishing} applied to 
$A=i(\Delta,\cF)(K_{\cF}+\Delta)$, 
since the required assumptions are verified as in the proof of Theorem~\ref{thm:Null} 
(see also \cite[Lemma~5.5]{JhL}).
\end{proof}

\subsection*{Acknowledgements}
The authors would like to thank Professors Shengli Tan, Xiaolei Liu, and Xin Lv for helpful discussions and suggestions. 
They are particularly grateful to Professor Jihao Liu for carefully reading earlier versions of the manuscript and providing valuable comments. 
The authors also thank the referees for their careful reading, for pointing out several errors, and for their many helpful suggestions, which have significantly improved the paper.

\end{document}